\documentclass[10pt,reqno]{amsart}

\usepackage{amsfonts,amsmath,amssymb,amsthm}
\usepackage{bbm}
\usepackage{hyperref}
\usepackage[english]{babel}
\usepackage{physics}
\usepackage{xcolor}

\newtheorem{theorem}{Theorem}[section]
\newtheorem{lemma}{Lemma}[section]
\newtheorem{proposition}{Proposition}[section]

\newtheorem{remark}{Remark}[section]
\newtheorem{definition}{Definition}[section]

\numberwithin{equation}{section}

\usepackage[a4paper,margin=3cm]{geometry}
\newcommand{\mres}{\mathbin{\vrule height 1.6ex depth 0pt width
		0.13ex\vrule height 0.13ex depth 0pt width 1.3ex}}

\newcommand{\Om}{\Omega}
\newcommand{\lbd}{\lambda}
\newcommand{\R}{\mathbb{R}}
\newcommand{\N}{\mathbb{N}}
\newcommand{\Hs}{\mathcal{H}}
\newcommand{\intzt}{\int_{0}^{t}}
\newcommand{\into}{\int_{\Om}}
\newcommand{\intpo}{\int_{\partial \Om}}
\newcommand{\intot}{\int_{t_1}^{t_2}}

\newcommand{\rk}{ \right\}}
\newcommand{\lk}{ \left\{}
\newcommand{\Div}{ {\rm div}}
\newcommand{\Q}{\mathcal{Q}}
\newcommand{\raw}{\rightarrow}
\newcommand{\M}{\mathbb{M}}
\newcommand{\sym}{{\rm sym}}

\newcommand{\res}{\mathbin{\vrule height 1.6ex depth 0pt width 0.13ex\vrule height 0.13ex depth 0pt width 1.3ex}}

\begin{document}
\title[Mixed boundary conditions as limits of dissipative boundary conditions]{Mixed boundary conditions as limits of dissipative boundary conditions in dynamic perfect plasticity}
	
\author[J.-F. Babadjian]{Jean-Fran\c cois Babadjian}
\address[Jean-Fran\c cois Babadjian]{Universit\'e Paris--Saclay, Laboratoire de Math\'ematiques d'Orsay, 91405, Orsay, France}
\email{jean-francois.babadjian@universite-paris-saclay.fr}

\author[R. Llerena]{Randy Llerena}
\address[Randy Llerena]{Research Platform MMM ``Mathematics-Magnetism-Materials" - Fak. Mathematik Univ. Wien, A1090 Vienna}
\email{randy.llerena@univie.ac.at}

\date{\today} 
\begin{abstract}
This paper addresses the well posedness of a dynamical model of perfect plasticity with mixed boundary conditions for general closed and convex elasticity sets. The proof relies on an asymptotic analysis of the solution of a perfect plasticity model with relaxed dissipative boundary conditions obtained in \cite{BC}. One of the main issues consists in extending the measure theoretic duality pairing between stresses and plastic strains, as well as a convexity inequality to a more general context where deviatoric stresses are not necessarily bounded. Complete answers are given in the pure Dirichlet and pure Neumann cases. For general mixed boundary conditions, partial answers are given in dimension $2$ and $3$ under additional geometric hypothesis on the elasticity set and the reference configuration.
\end{abstract}
	
	
\maketitle
	
	
\section{Introduction}
Elasto-plasticity is a classical theory of continuum mechanics \cite{H,Lj} that predicts the appearance of permanent deformations in materials when an internal critical stress is reached. At the atomistic level, these plastic deformations occur  when the crystal lattice of the atoms are misaligned due to the accumulation of slips defects, called dislocations. These dislocations determine the change of behavior of a body from an elastic and reversible state to a plastic and irreversible one.
		
At the continuum level, and in the context of small deformations, the theory involves the {\it displacement field} $u:\Omega \times (0,T) \to \R^n$ and the {\it Cauchy stress tensor} $\sigma: \Omega \times (0,T) \to \M^n_\sym$, both defined on the reference configuration $\Omega$ of the body, a bounded open subset of $\R^n$ ($n=2,3$). They first satisfy the {\it equation of motion}
\begin{equation}
\ddot{u} - \Div \sigma = f \quad \text{ in }\Om \times (0,T), \label{eq:eqmotion}
\end{equation}
for some (given) external body load $f : \Omega \times (0,T) \to \R^n$. In the previous expression, and in the sequel, the dot stands for the partial derivative with respect to time. One particular feature of perfect plasticity is that the stress tensor is constrained to take its values into a fixed closed and convex set $\mathbf K$ of the space $\M^n_\sym$ of symmetric $n \times n$ matrices, also called {\it elasticity set}:
\begin{equation}
\label{eq:stressconstraint}
\sigma \in {\bf K}.
\end{equation}
In classical elasticity, the linearized strain is purely elastic and it is represented by the symmetric part of the gradient of displacement, i.e.  $Eu:=(Du+Du^T)/2$. In perfect elasto-plasticity, the elastic strain $e:\Omega \times (0,T) \to \M^n_\sym$ only represents a part of the linearized strain $Eu$. It stands for the reversible part of the total deformation and it is related to $\sigma$ by means of {\it Hooke's law}, which we assume to be isotropic:
\begin{equation}\label{eq:Hooke}
\sigma=\mathbf A e=\lambda (\tr e) {\rm Id} + 2\mu e,
\end{equation}
for some constants $(\lambda,\mu) \in \R^2$, called Lam\'e coefficients, which satisfy the ellipticity conditions $\mu>0$ and $n\lambda+2\mu>0$. The remaining part of the strain, 
\begin{equation}\label{eq:additive}
p:=Eu-e
\end{equation}
stands for the plastic strain leading to irreversible deformations. It is a new unknown of the problem whose evolution is described by means of a {\it flow rule}. Assuming that $\mathbf K$ has nonempty interior, it stipulates that if $\sigma$ belongs to the interior of $\mathbf K$, then the material behaves elastically and no additional inelastic strains are created, i.e. $\dot p=0$. On the other hand, if $\sigma$ reaches the boundary of $\mathbf K$, then $\dot p$ may develop in such a way that a non–trivial permanent plastic strain $p$ may remain after unloading. The evolution of $p$ is described by the {\it Prandtl-Reuss law} 
$$\dot p \in N_{\mathbf K}(\sigma),$$
where $N_{\mathbf K}(\sigma)$ stands for  the normal cone to $\mathbf K$ at $\sigma$, or equivalently, thanks to convex analysis, by  {\it Hill’s principle of maximum plastic work}
\begin{equation}
H(\dot p) = \sigma : \dot p \label{eq:hillslaw},
\end{equation}
where $H(q) := \sup_{\tau \in {\bf K}} \tau : q$ is the support function ${\bf K}$. The system \eqref{eq:eqmotion}--\eqref{eq:hillslaw} has to be supplemented by initial conditions
\begin{equation}\label{eq:initialcondition}
(u(0), \dot{u}(0), e(0), p(0)) = (u_0, v_0,e_0,p_0)
\end{equation} 
as well as suitable boundary conditions to be discussed later, and which will be one of the main focus of this work.

For most of metals and alloys, standard models of perfect plasticity involve elasticity sets $\mathbf K$ which are invariant in the direction of hydrostatic matrices (multiples of the identity) and bounded in the direction of deviatoric (trace free) ones. This is for example the case of the Von Mises and Tresca models (see e.g. \cite{A1,AG,Tem85} in the static case,  \cite{A2,FG,Suquet,DMDSM} in the quasi-static case and \cite{AL,DS} in the dynamic one). In other situations like in the context of soils mechanics, it is of importance to consider elasticity sets $\mathbf K$ that are not necessarily invariant with respect to hydrostatic matrices. So called Drucker-Prager or Mohr-Coulomb models fall within this framework (see \cite{BC,BMo,BMo2}). In this paper, we treat as utmost as possible the case of a general elasticity set ${\bf K}$.

Let us now discuss the boundary conditions. Having in mind that the system of dynamic elasto-plasticity described so far has a hyperbolic nature, one has to consider boundary conditions compatible with this hyperbolic structure, in particular, with the finite speed propagation of the initial data along the characteristic lines. A general approach to this type of initial--boundary value constrained hyperbolic systems has been studied in \cite{DMS} (see also \cite{DLS}) where a class of so-called admissible dissipative boundary conditions has been introduced. This problem has subsequently been specified to the case of plasticity, first in \cite{BM} for a simplified scalar model, and then in \cite{BC} for the general vectorial model as described before. In this context, all admissible (homogeneous) {\it dissipative boundary conditions} take the form (see \cite[Section 3]{BC})
\begin{equation}
\label{eq:dissipativebdr}
S \dot u + \sigma \nu = 0 \quad \text{ on }\partial\Omega \times (0,T),
\end{equation}
where $\nu$ denotes the outer unit normal to $\Om$, and $S:\partial \Omega \to \mathbb{M}^n_\sym$ is a  spatially dependent positive definite boundary matrix. The well posedness of the initial--boundary value system \eqref{eq:eqmotion}--\eqref{eq:dissipativebdr} has been carried out in \cite{BC}. It has been established existence and uniqueness of two equivalent notions of relaxed solutions (variational and entropic solutions).  The relaxation phenomena is a simple consequence of the fact that, formally, the stress contraint \eqref{eq:stressconstraint} might not be compatible with the boundary condition \eqref{eq:dissipativebdr}. Indeed, if $\sigma(t) \in \mathbf K$ in $\Omega$, we would expect that $\sigma(t)\nu \in \mathbf K\nu$ on $\partial\Omega$ while $\sigma(t)\nu=-S \dot u(t)$ is free on the boundary. Thus, the boundary condition and the stress constraint have to accomodate to each other and the dissipative boundary condition \eqref{eq:dissipativebdr} has to be relaxed into 
\begin{equation}
\label{eq:relaxedBC}
P_{-\mathbf K\nu}(S \dot u)+\sigma\nu=0 \quad \text{ on }\partial\Omega \times (0,T),
\end{equation}
where, for $x \in \partial\Om$, $P_{-\mathbf K\nu(x)}$ stands for the orthogonal projection in $\R^n$ onto the convex set $-\mathbf K \nu(x)$ with respect to a suitable scalar product. This is indeed a relaxation in the sense of the Calculus of Variations, because the energy balance involves a term of the form
$$\int_\Omega H(\dot p)\, dx + \frac12 \int_{\partial\Omega} S \dot u\cdot \dot u\, d\mathcal H^{n-1} + \frac12 \int_{\partial\Omega}S^{-1}(\sigma\nu)\cdot (\sigma\nu)\, d\mathcal H^{n-1}.$$
The previous energy functional turns out of not being lower semicontinuous with respect to weak convergence in the energy space, and its relaxation with respect to this topology is explicitly given by
$$\int_\Omega H(\dot p)\, dx +  \int_{\partial\Omega} \psi(x,\dot u)\, d\mathcal H^{n-1} + \frac12 \int_{\partial\Omega}S^{-1}(\sigma\nu)\cdot (\sigma\nu)\, d\mathcal H^{n-1},$$
where $\psi(x,\cdot)$ is the inf-convolution of the functions $z \mapsto \frac12 S(x)z \cdot z$ and $z \mapsto H(-z\odot \nu(x))$. The connexion between the relaxed energy and the modified boundary condition \eqref{eq:relaxedBC} comes from a first order minimality condition and the following formula (see \cite[Section 4]{BC})
$$D_z\psi(x,\dot u(t,x))= P_{-\mathbf K\nu(x)}(S(x)\dot u(t,x)).$$

Unfortunately, Dirichlet, Neumann and mixed boundary conditions are not admissible because the matrix $S$ is not allowed to vanish nor to take the value $\infty$. It is the main focus of the present work to show that these type of natural boundary conditions can actually be obtained by means of an asymptotic analysis letting $S \to \infty$ in a portion of the boundary where we want to recover a Dirichlet  condition, and letting $S \to 0$ on the complementary part where one wishes to formulate a Neumann condition. This type of analysis has already been performed in \cite{BM} in the simplified case of antiplane scalar plasticity where pure Dirichlet and pure Neumann boundary conditions have been derived. We extend here this analysis to the general vectorial case where additional issues arise, and to the case of mixed boundary conditions.

To be more precise, in the spirit of \cite{FG,KT, DMDSM}, we partition $\partial \Om$ into the disjoint union of $\Gamma_D, \Gamma_N$ and $\Sigma$, where $\Gamma_D$ and $\Gamma_N$ stand for the Dirichlet and Neumann parts of the boundary, respectively, and  $\Sigma$ is the interface between $\Gamma_D$ and $\Gamma_N$ which is $\mathcal H^{n-1}$-negligible. We consider a boundary matrix of the form
\begin{equation}\label{eq:Slambda}
S_\lambda(x) := \left( \lbd \textbf{1}_{\Gamma_D} + \frac{1}{\lbd} \textbf{1}_{\Gamma_N} \right)\textrm{Id}
\end{equation}
for some parameter $\lbd >0$ which will be sent to $\infty$. Denoting by $(u_\lambda,e_\lambda,p_\lambda,\sigma_\lambda)$ the unique weak solutions of the system \eqref{eq:eqmotion}--\eqref{eq:initialcondition} with the relaxed dissipative boundary condition \eqref{eq:relaxedBC} associated to the boundary matrix $S_\lambda$, using the results of \cite{BC}, we easily derive bounds in the energy space for this quadruple, which allow one to get weak limits $(u,e,p,\sigma)$ and pass to the limit into the equation of motion \eqref{eq:eqmotion}, the stress constraint \eqref{eq:stressconstraint}, Hooke's law \eqref{eq:Hooke}, the additive decomposition \eqref{eq:additive} and the initial condition \eqref{eq:initialcondition}. This is the object of Lemma \ref{lem:compactness}. As usual in plasticity, the main difficulty consists in passing to the limit in the flow rule expressed by \eqref{eq:hillslaw} and in the relaxed boundary condition \eqref{eq:relaxedBC}. In accordance with \cite{BC,BM,BMo}, the idea consists in taking the limit as $\lambda\to\infty$ into the energy balance. The main difficulty is concerned with the term
$$\int_\Omega H(\dot p_\lambda)\, dx +  \int_{\partial\Omega} \psi_\lambda(x,\dot u_\lambda)\, d\mathcal H^{n-1} + \frac12 \int_{\partial\Omega}S_\lambda^{-1}(\sigma_\lambda\nu)\cdot (\sigma_\lambda\nu)\, d\mathcal H^{n-1},$$
where
$$\psi_\lbd (x,z): = \inf_{w \in \R^n}{\left\{ \frac{1}{2} \left( \lbd {\bf 1}_{\Gamma_D} + \frac 1 \lbd {\bf 1}_{\Gamma_N} \right)|w|^2  + H((w- z) \odot \nu (x) ) \right\} }.$$
A uniform bound on the previous energy easily shows that
$$\int_{\Gamma_N} |\sigma_\lambda\nu|^2\, d\mathcal H^{n-1} \leq \frac{C}{\lambda} \to 0, \quad \text{ as }\lambda \to \infty,$$
which leads to the Neumann boundary condition $\sigma\nu=0$ on $\Gamma_N$. The obtention of the Dirichlet boundary condition on $\Gamma_D$ is more involved because, as usual in perfect plasticity, concentration phenomena might occur. A convex analysis argument based on the Moreau-Yosida approximation of $H$ yields the following lower bound on the energy (see Lemma \ref{prop:relaxationdirichletpart})
$$\int_{\Om}H( \dot p)\, dx +\int_{{\Gamma_D}}H(- \dot u \odot \nu)\, d\mathcal H^{n-1}\le \liminf_{\lbd \rightarrow \infty}{\left( \int_\Om H( \dot p_\lambda)\, dx + \int_{\partial \Om}{\psi_\lbd(x, \dot u_\lbd)\, d \Hs^{n-1}} \right)}.$$
Proving that this lower bound is also an upper bound is formally a consequence the convexity inequality 
$$H(\dot p) \geq \sigma:\dot p$$ 
(because $\sigma \in \mathbf K$), and integrations by parts in space and time. Unfortunately, this formal convexity inequality is very difficult to justify in the context of perfect plasticity because the Cauchy stress $\sigma$ and the plastic strain rate $\dot p$ are not in duality. Indeed, the natural energy space gives $\sigma(t) \in H(\Div,\Om)$ while $\dot p((t) \in \mathcal M(\Om \cup \Gamma_D;\M^n_\sym)$ since the support function $H$ grows linearly with respect to its argument. In particular, the plastic dissipation
$$\int_\Om H(\dot p(t))\, dx$$
has to be understood as a convex function of a measure (see \cite{DT1,DT2,GS}). Whenever the quadruple $(u,e,p,\sigma)$ belongs to the energy space, it follows that $(\dot u(t),\dot e(t),\dot p(t))$ belongs to the space of all kinematically admissible triples
\begin{multline*}
 \Big\{(v,\eta,q) \in [BD(\Om) \cap L^2(\Om;\R^n)] \times L^2(\Om;\mathbb M^n_{\rm sym}) \times \mathcal M(\Om \cup \Gamma_D ;\mathbb M^n_{\rm sym}) : \\
 Ev=\eta+q \text{ in }\Om, \quad q=-v\odot \nu\mathcal H^{n-1} \text{ on }\Gamma_D\Big\},
\end{multline*}
and $\sigma(t)$ belongs to the space of all statically and plastically admissible stresses
$$\{\tau \in H(\Div,\Om): \; \tau\nu=0 \text{ on }\Gamma_N, \; \tau(x) \in \mathbf K \text{ a.e. in }\Om\}.$$
In the spirit of \cite{ FG,KT, DMDSM}, it allows one to consider a generalized stress/strain duality (see Definition \ref{definition:dualityMixed}) as the first order distribution $[\sigma(t) \colon \dot p(t)] \in \mathcal{D}' (\R^n)$, compactly supported in $\overline\Om$, defined as
\begin{equation}
\label{eq:introstrain}
\langle [\sigma (t)\colon \dot p(t)], \varphi \rangle = -\int_\Omega \varphi \sigma(t) : \dot e(t) \, dx- \int_\Omega \dot u(t) \cdot {\rm div} \sigma(t) \, \varphi \, dx - \int_\Omega \sigma(t) \colon \big(u(t) \odot \nabla \varphi\big) \, dx
\end{equation}
or any $\varphi \in C^\infty_c (\R^n)$. The question now reduces to prove that 
\begin{equation}
\label{eq:introconvexinea}
H(\dot p(t)) \ge [\sigma(t) \colon \dot p(t)] \quad \textrm{in } \mathcal M (\R^n),
\end{equation}
and this is the object of Section \ref{sec:duality}	. In Propositions \ref{prop:an1} we show that this generalized convexity inequality is always satisfied in the pure Dirichlet ($\Gamma_D=\partial\Om$) and pure Neumann ($\Gamma_N=\partial\Om$) cases. In the case of mixed boundary conditions, there might be some concentration effects at the interface $\Sigma$ between the Dirichlet and the Neumann parts, and the previous convexity inequality is shown to hold only in $\mathcal M (\R^n \setminus \Sigma)$ in Proposition  \ref{prop:convexinequalityMixed}. Unfortunately, this weaker result is not enough to conclude the energy  upper bound because, although $\Sigma$ is $\mathcal H^{n-1}$-negligible, some undesirable energy concentration might accumulate on that set. We further exhibit special cases in dimensions $n=2$ and $n=3$ which guarantee the validity of \eqref{eq:introconvexinea} also in the case of mixed boundary conditions (see Propositions \ref{prop:n=2} and \ref{prop:n=3}). In dimension $n=2$, it is enough to assume that $\Sigma$ is a finite set (as in \cite{FG}) while in dimension $n=3$, we suppose that the convex set $\mathbf K$ is invariant in the direction of hydrostatic matrices and bounded in the direction of deviatoric ones, as well as additional regularity assumptions on the reference configuration $\Om$ (as in \cite{KT}).
			
To conclude this introduction, let us mention that our method only allows one to derive homogeneous mixed boundary conditions. Indeed, at a formal level, even starting from a nonhomogeneous dissipative boundary condition of the form $S\dot u+\sigma\nu=g$ on $\partial\Om \times (0,T)$, for some non trivial source term $g$, (or its relaxed counterpart $P_{-\mathbf K\nu}(S\dot u-g)+\sigma\nu=0$ on $\partial\Om \times (0,T)$ given by an adaptation of \cite{BC}), we obtain an energy balance involving the following additional term
$$\int_0^T \int_{\partial \Om}S^{-1} g \cdot g \, d \Hs^{n-1}\, dt .$$
Specializing the problem to a boundary matrix $S=S_\lambda$ of the form \eqref{eq:Slambda} and some $\lambda$-dependent source term $g_\lambda \in   L^2 (\partial \Om \times (0,T); \R^n)$, the previous discussion shows that a uniform bound on the solution $(u_\lambda,e_\lambda,p_\lambda,\sigma_\lambda)$ in the energy space would require that 
$$\sup_{\lambda>0} \left\{\frac{1}{\lambda} \int_{\Gamma_D} |g_\lambda|^2\, d\mathcal H^{n-1} + \lambda \int_{\Gamma_N} |g_\lambda|^2\, d\mathcal H^{n-1} \right\}<\infty.$$
It would imply that 
$$\sigma_\lambda\nu=g_\lambda-\lambda^{-1} \dot u_\lambda \to 0 \quad \text{ in }\Gamma_N \times (0,T)$$ in a weak sense as $\lambda \to \infty$ (because the trace of $\dot u_\lambda$ is bounded in $L^1(\partial\Om \times (0,T);\R^n)$), leading to a homogenous Neumann condition in $\Gamma_N$. Concerning the Dirichlet part, formally reporting this information in the  dissipative boundary condition restricted to $\Gamma_D$ would lead to
$$\dot u_\lambda = \lambda^{-1}g_\lambda-\lambda^{-1}\sigma_\lambda\nu \to 0 \quad \text{ in }\Gamma_D \times (0,T),$$
in some weak sense 	as $\lambda \to \infty$ (because $\sigma_\lambda\nu$ is bounded in $L^2(0,T;H^{-1/2}(\partial\Om;\R^n))$), leading to a homogeneous Dirichlet boundary condition. Strictly speaking one should rather consider the relaxed boundary condition which would lead to a strain concentration on $\Gamma_D$ associated to a homogeneous Dirichlet boundary condition.			
			
The paper is organized as follows. In Section 2, we introduce various notation and basic facts used throughout this paper. In  Section 3, we discuss the notion duality between plastic strains and Cauchy stresses, and we prove generalized convexity inequalities of the form \eqref{eq:introconvexinea} involving these two arguments which are not in duality in the energy space. Finally, in Section 4, we state and prove our main result, Theorem \ref{thm:compactness}, about the convergence of the solutions obtained in \cite{BC} to the (unique) solution of a dynamical elasto-plastic model with homogeneous mixed boundary conditions.

\section{Notation and preliminaries}
	
\subsection{Linear algebra}
	
If $a$ and $b \in \R^n$, we write $a \cdot b:=\sum_{i=1}^n a_i b_i$ for the Euclidean scalar product, and we denote by $|a|:=\sqrt{a \cdot a}$ the corresponding norm. 
	
We denote by $\mathbb M^n$ the set of $n \times n$ matrices and by $\mathbb M^{n}_{\rm sym}$ the space of symmetric $n \times n$ matrices. The set of all (deviatoric) trace free symmetric matrices will be denoted as $\mathbb M^{n}_{D}$. The space $\mathbb M^n$ is endowed with the Fr\"obenius scalar product $A:B:=\tr(A^T B)$ and with the corresponding Fr\"obenius norm $|A|:=\sqrt{A:A}$. If $a \in \R^n$ and $b \in \R^n$, we denote by $a \odot b:=(ab^T+b^T a)/2 \in  \mathbb M^{n}_{\rm sym}$ there symmetric tensor product.
	
If $A \in \mathbb M^{n}_{\rm sym}$, there exists an orthogonal decomposition of $A$ with respect to the Fr\"obenius scalar product as follows
$$ A = A_D + \frac{1}{n} (\tr A){\rm Id} ,$$
where $A_D \in \mathbb M^n_D$ stands for the deviatoric part of $A$. 
	
\subsection{Measures}

The Lebesgue measure in $\R^n$ is denoted by $\mathcal L^n$, and the $(n-1)$-dimensional Hausdorff measure by $\mathcal H^{n-1}$. If $X \subset \R^n$ is a Borel set and $Y$ is an Euclidean space, we denote by $\mathcal M(X;Y)$ the space of $Y$-valued bounded Radon measures in $X$ endowed with the norm $\|\mu\|:=|\mu|(X)$, where $|\mu|$ is the variation of the measure $\mu$. If $Y=\R$ we simply write $\mathcal M(X)$ instead of $\mathcal M(X;\R)$.
	
If the relative topology of $X$ is locally compact, by Riesz representation theorem, $\mathcal M(X;Y)$ can be identified with the dual space of $C_0(X;Y)$, the space of continuous functions $\varphi:X \to Y$ such that $\{|\varphi|\geq \varepsilon\}$ is compact for every $\varepsilon>0$. The (vague) weak* topology of $\mathcal M(X;Y)$ is defined using this duality. 
	
Let $\mu \in \mathcal M(X;Y)$ and $f:Y \to [0,+\infty]$ be a convex, positively one-homogeneous function.  Using the theory of convex functions of measures developed in \cite{DT1,DT2, GS}, we introduce the nonnegative {Borel measure $f(\mu)$}, defined by 
$$f(\mu)=f\left(\frac{d \mu}{d |\mu|}\right)|\mu|\,,$$
where $\frac{d \mu}{d |\mu|}$ stands for the Radon-Nikod\'ym derivative of $\mu$ with respect to $|\mu|$.
	
\subsection{Functional spaces}
	
We use standard notation for Lebesgue spaces ($L^p$) and Sobolev spaces ($W^{s,p}$ and $H^s=W^{s,2}$).
	
The space of functions of bounded deformation is defined by
$$BD(\Omega)=\{u \in L^1(\Omega;\R^n) : \; E u \in \mathcal M(\Omega;\mathbb M^n_{\rm sym})\}\,,$$
where $E u:=(Du+Du^T)/2$ stands for the distributional symmetric gradient of $u$. We recall (see \cite{Bab, Tem85}) that, if $\Omega$ has a Lipschitz boundary, every function $u \in BD(\Omega)$ admits a trace, still denoted by $u$, which belongs to $L^1(\partial\Omega;\R^n)$, and such that the integration by parts formula holds: for all $\varphi \in C^1(\overline \Omega;\mathbb M^n_{\rm sym})$,
$$\int_{\partial\Omega} u\cdot (\varphi\nu)\, d\mathcal H^{n-1}=\int_\Omega {\rm div} \varphi \cdot u\, dx + \int_\Omega \varphi:d E u\,.$$
Note that the trace operator is continuous with respect to the strong convergence of $BD(\Omega)$ but not with respect to the weak* convergence in $BD(\Omega)$.
	
Let us define 
$$H({\rm div},\Omega)=\{\sigma \in L^2(\Omega;\mathbb M^n_{\rm sym}) :\;  {\rm div} \sigma \in L^2(\Omega;\R^n)\}\,.$$
If $\Omega$ has Lipschitz boundary, for any $\sigma \in H(\mathrm{div}, \Omega)$ we can define the normal trace $\sigma\nu$ as an element of $H^{-\frac12}(\partial \Omega;\R^n)$  (cf. e.g.\ \cite[Theorem~1.2, Chapter~1]{Tem85}) by setting	\begin{equation}\label{2911181910}
\langle \sigma \nu, \psi \rangle_{H^{-\frac12}(\partial \Omega;\R^n),H^{\frac12}(\partial \Omega;\R^n)}:= \int_\Omega \psi \cdot {\rm div} \sigma \, dx + \int_\Omega \sigma \colon  E \psi \, dx\,.
\end{equation}
for every $\psi \in H^1(\Omega;\R^n)$. 
		
\section{Duality between stress and plastic strain}\label{sec:duality}	
	
In the spirit of \cite{KT,FG,BMo}, we define a generalized notion of stress/strain duality.
	
\noindent $(H_1)$ {\bf The reference configuration.} Let $\Om \subset \R^n$ be a bounded open set with Lipschitz boundary. We assume that $\partial \Om$ is decomposed as the following disjoint union	
$$\partial \Om = \Gamma_D \cup \Gamma_N \cup \Sigma,$$
where $\Gamma_D$ and $\Gamma_N$ are open sets in the relative topology of $\partial\Om$, and $\Sigma = \partial _{| \partial \Om}\Gamma_D = \partial _{| \partial \Om}\Gamma_N$ is $\Hs^{n-1}$-negligible. 
	
On the Neumann part $\Gamma_N$, we will prescribe a surface load given by a function $g \in L^\infty(\Gamma_N;\R^n)$. The space of {\it statically admissible stresses} is defined by
$$\mathcal S_g:=\{\sigma \in H({\rm div},\Omega) : \;\sigma\nu=g \text{ on }\Gamma_N\}.$$
In the sequel we will also be interested in stresses $\sigma$ taking values in a given set.

\noindent $(H_2)$ {\bf Plastic properties.} Let $\mathbf K  \subset  \mathbb{M}^n_{{\rm sym}}$ be a closed convex set such that $0$ belongs to the interior point of ${\bf K}$. In particular, there exists $r>0$ such that
\begin{equation}
\label{eq:ms2}
\lk \tau \in \mathbb{M}^n_{{\rm sym}} : \abs{\tau} \le r \rk \subset {\bf K}.
\end{equation}
The support function $H : \mathbb{M}^n_{{\rm sym}} \raw \left[ 0, + \infty \right]$ of $\bf K$ is defined by
$$H(q) := \sup_{\sigma \in {\bf K}}{\sigma : q} \quad \textrm{for all } q \in \mathbb{M}^n_{{\rm sym}}.$$
We can deduce from (\ref{eq:ms2}) that
\begin{equation}\label{eq:coercH}
H(q) \ge r \abs{q}  \quad \textrm{for all } q \in \mathbb{M}^n_{{\rm sym}}.
\end{equation}
If $p \in \mathcal{M} (\Om \cup \Gamma_D;  \mathbb{M}^n_{{\rm sym}}),$ we denote the convex function of a measure $H(p)$ by
$$H(p) := H \left( \frac{dp}{d \abs{p}} \right) \abs{p},$$
and the plastic dissipation is defined by
$$\mathcal{H} (p) := \int_{\Om \cup \Gamma_D} H\left( \frac{dp}{d \abs{p}} \right) d\abs{p}. $$
We define the set of all {\it plastically admissible stresses} by
$$\mathcal K:=\{ \sigma \in H(\Div,\Om) \colon \, \sigma(x) \in \mathbf K \text{ for a.e.\ } x \in \Omega \}$$
which defines a closed and convex subset of $H(\Div,\Om) $.   
	
The portion $\Gamma_D$ of $\partial\Om$ stands for the Dirichlet part of the boundary where a given displacement $w$ will be prescribed. We assume that it extends into a function $w \in H^1(\Om;\R^n)$ (so that $w_{|\Gamma_D}\in H^{1/2}(\Gamma_D;\R^n)$). We define the space of {\it kinematically admissible triples} by
\begin{eqnarray*}
\mathcal A_w & := & \Big\{(u,e,p) \in [BD(\Om) \cap L^2(\Om;\R^n)] \times L^2(\Om;\mathbb M^n_{\rm sym}) \times \mathcal M(\Om \cup \Gamma_D ;\mathbb M^n_{\rm sym}) : \\
&& Eu=e+p \text{ in }\Om, \quad p=(w-u)\odot \nu\mathcal H^{n-1} \text{ on }\Gamma_D\Big\},
\end{eqnarray*}
where $\nu$ is the outer unit normal to $\Om$. The function $u$ stands for the displacement, $e$ is the elastic strain and $p$ is the plastic strain. The following result provides an approximation for triples $(u,e,p) \in \mathcal A_w$ and its proof follows the line of Step 1 in \cite[Theorem 6.2]{FG}.

\begin{lemma}
\label{lem:approximationBDkin}
Let $(u,e,p) \in [BD(\Om) \cap L^2(\Om; \R^n)] \times L^2(\Om; \M^n_\sym) \times \mathcal{M} (\Om; \M^n_\sym)$ be such that $Eu = e+ p$ in $\Om$. Then, there exists a sequence $\{(u_k,e_k, p_k)\}_{k \in \N}$ in $C^{\infty} (\overline{\Om}; \R^n \times \M^n_\sym \times \M^n_\sym) $ such that
$$E u_k = e_k + p_k \quad \text{ in }\Om,$$
\begin{equation}\label{eq:approximationuep}
\begin{cases}
u_k \raw u \quad \textrm{strongly in } L^2( \Om; \R^n), \\
e_k \raw e \quad \textrm{strongly in } L^2( \Om; \M^n_\sym), \\
p_k \rightharpoonup p \quad \textrm{weakly* in } \mathcal{M} (\Om; \M^n_\sym),\\
|p_k|(\Omega) \to |p|(\Om),\\
|Eu_k|(\Omega) \to |Eu|(\Om),\\
u_k \to u \quad \textrm{strongly in }L^1(\partial\Omega;\R^n).
\end{cases}
\end{equation}
and for all $\varphi \in C^\infty_c(\R^n)$ with $\varphi \geq 0$,
\begin{equation}\label{eq:Res}
\limsup_{k \to \infty}\int_\Om \varphi\, dH(p_k)\leq \int_\Om \varphi\, dH(p).
\end{equation}
\end{lemma}
	
\begin{proof}
The construction of a sequence $\{(u_k,e_k, p_k)\}_{k \in \N}$ in $C^{\infty} (\overline{\Om}; \R^n \times \M^n_\sym \times \M^n_\sym) $ such that $E u_k = e_k + p_k$ in $\Om$ together with the four first convergences of \eqref{eq:approximationuep} result from Step 1 in \cite[Theorem 6.2]{FG}. Moreover, a careful inspection of that proof also shows that $|Eu_k|(\Omega) \to |Eu|(\Om)$. The strong convergence of the trace in $L^1(\partial\Omega;\R^n)$ is a consequence of \cite[Proposition 3.4]{Bab}. The last condition \eqref{eq:Res} follows as well from the proof of \cite[Theorem 6.2]{FG} using the subadditivity and the positive one-homogeneity of $H$. Note that \eqref{eq:Res} cannot be directly obtained from the strict convergence of $\{p_k\}_{k \in \mathbb N}$ and Reshetnyak continuity Theorem (see \cite[Theorem 2.39]{AFP} or \cite{S}) because $H$ is just lower semicontinuous and it can take infinite values.
\end{proof}
	
We now define a distributional duality pairing between statically admissible stresses and plastic strains.	
	
\begin{definition}
\label{definition:dualityMixed}
Let $\sigma \in \mathcal S_g$ and $(u,e,p)\in \mathcal A_w$. We define the first order distribution 
$[\sigma \colon p]\in \mathcal D'(\R^n)$ by
$$\langle [\sigma \colon p], \varphi \rangle :=\int_\Omega \varphi \sigma : (E w-e)\, dx+\int_\Omega (w-u) \cdot {\rm div} \sigma \, \varphi \, dx + \int_\Omega \sigma \colon \big((w-u) \odot \nabla \varphi\big) \, dx + \int_{\Gamma_N}{\varphi g \cdot u \, d \Hs^{n-1} }$$
for all $\varphi \in C^\infty_c(\R^n)$.
\end{definition}
	
\begin{remark}
{\rm If $\varphi \in C^\infty_c(\Omega)$, thanks to the integration by parts formula in $H^1(\Omega;\R^n)$, the expression of the stress/strain duality becomes independent of $w$ and $g$, and it reduces to
\begin{equation}\label{1410191537}
\langle [\sigma \colon p], \varphi \rangle= -\int_\Omega \varphi \sigma \colon e \,dx- \int_\Omega u\cdot {\rm div} \sigma \, \varphi \,dx - \int_\Omega \sigma \colon (u\odot \nabla \varphi) \, dx\,.
\end{equation}
}
\end{remark}

As already observed in \cite{BMo}, contrary to \cite{FG,KT}, we are not able to show in general that $[\sigma:p]$ extends into a bounded Radon measure. This is due to the fact that, in our context, $\sigma_D$ fails to belong to $L^\infty(\Om;\M^n_D)$. However, provided $\sigma \in \mathcal K$ and under suitable assumption on $\Om$ and $\mathbf K$, we are going to show a convexity inequality which will ensure that $H(p)-[\sigma:p]$ is a nonnegative distribution, hence that $[\sigma:p]$ actually defines a bounded Radon measure supported in $\overline\Om$. 
	
\subsection{Pure Dirichlet or pure Neumann boundary conditions}
	
As the following result shows, the distribution $[\sigma:p]$ always extends into a bounded Radon measure in the pure Dirichlet or pure Neumann cases.
\begin{proposition}
\label{prop:an1}
Let $\Om \subset \R^n$ be a bounded open set with Lipschitz boundary. Assume that either $\partial\Om=\Gamma_D$ or $\partial\Om=\Gamma_N$. Then, for every $\sigma \in \mathcal S_g \cap \mathcal K$ and $(u,e,p) \in \mathcal A_w$ with $H(p) \in \mathcal M(\Om \cup \Gamma_D)$, the distribution $[\sigma:p]$ extends to a bounded Radon measure supported in $\overline \Om$ and 
\begin{equation}
\label{eq:conv-ineqNeumann}
H(p) \ge \left[ \sigma : p \right] \quad {{\rm in }} \ \mathcal{M}(\R^n).
\end{equation}
\end{proposition}
	
\begin{proof} In the case of pure Dirichlet boundary conditions, $\partial\Om=\Gamma_D$, we first note that $\mathcal S_g=H({\rm div},\Om)$. The duality pairing is then independent of $g$ and reduces to
$$\langle [\sigma \colon p], \varphi \rangle = \int_\Omega \varphi \sigma : (E w-e)\, dx+ \int_\Omega (w-u) \cdot {\rm div} \sigma \, \varphi \,dx + \int_\Omega \sigma \colon \big((w-u) \odot \nabla \varphi\big) \,dx$$
for all $\varphi \in C^\infty_c(\R^n)$. This case has already been addressed in \cite[Section 2]{BMo}. The result is a direct consequence an approximation result for $\sigma \in \mathcal K$ by smooth functions (see e.g. \cite[Lemma 2.3]{DMDSM}) as well as the integration by parts formula in $BD(\Omega)$ (see \cite[Theorem 3.2]{Bab}).
		
\medskip
		
In the case of pure Neumann boundary conditions, $\partial\Om=\Gamma_N$, using the integration by parts formula in $H^1(\Om;\R^n)$ for the function $w$, the duality pairing becomes independent of $w$ and reduces to
$$\langle [\sigma \colon p], \varphi \rangle := -\int_\Omega \varphi \sigma : e \,dx- \int_\Omega u \cdot {\rm div} \sigma \, \varphi \,dx - \int_\Omega \sigma \colon \big(u \odot \nabla \varphi\big)\, dx+\int_{\partial\Omega} { \varphi} g \cdot u \, d\mathcal H^{n-1}$$
for all $\varphi \in C^\infty_c(\R^n)$. According to {Lemma \ref{lem:approximationBDkin}}, there exists a sequence $\{(u_k,e_k, p_k)\}_{k \in \N}$ in $C^{\infty} (\overline{\Om}; \R^n \times \M^n_\sym \times \M^n_\sym) $ such that $E u_k = e_k + p_k$ in $\Om$ and \eqref{eq:approximationuep}--\eqref{eq:Res} hold. By definition of the duality pairing $\left[ \sigma: p_k  \right]$, for all $\varphi \in C^\infty_c (\R^n)$ we have
\begin{equation}
\left< \left[ \sigma: p_k  \right], \varphi \right> := - \into{\sigma : e_k \varphi \, dx}  -\into{\varphi u_k \cdot \Div \sigma \, dx} - \into{\sigma : (u_k \odot \nabla \varphi) \, dx} +{ \int_{\partial\Omega} { \varphi} g \cdot u_k \, d\mathcal H^{n-1} }, \label{eq:an4}
\end{equation}
and using the integration by parts formula \eqref{2911181910} for $\sigma \in H({\rm div},\Om)$, we get that 
\begin{equation}
\left< \left[ \sigma: p_k  \right], \varphi \right> := \int_\Om \sigma : p_k \varphi \, dx. \label{eq:an4bis}
\end{equation}
By definition of the support function $H$, we have that $H(p_k) \geq \sigma:p_k$ a.e. in $\Om$, hence if $\varphi \geq 0$, by \eqref{eq:an4}, it yields
\begin{eqnarray*}
\int_\Om H(p_k)\varphi\, dx & \geq & \int_\Om \sigma:p_k \varphi \,dx \\
& =&  - \into{\sigma : e_k \varphi \, dx}-\into{\varphi u_k \cdot \Div \sigma \, dx} - \into{\sigma : (u_k \odot \nabla\varphi) \, dx}   { + \int_{\partial \Om}{\varphi g \cdot u_k \, d\Hs^{n-1}}}.
\end{eqnarray*}
Hence, passing to the limit as $k \to \infty$ thanks to the convergences  \eqref{eq:approximationuep}--\eqref{eq:Res} yields
\begin{eqnarray*}
\int_\Om \varphi\, dH(p) & \ge & - \into{\sigma : e \varphi \, dx}  -\into{\varphi u \cdot {\rm div} \sigma \, dx} - \into{\sigma : (u \odot \nabla \varphi) \, dx} {+ \int_{\partial \Om}{\varphi g \cdot u \, d\Hs^{n-1}}} \\
& =: & \left<\left[ \sigma: p  \right], \varphi \right>,
\end{eqnarray*}
where we used once more the definition of duality $[\sigma:p]$. As a consequence, the distribution $H(p)-[\sigma:p]$ is nonnegative, hence it extends into a bounded Radon measure in $\R^n$. Thus, $[\sigma:p]$ extends as well into a bounded Radon measure in $\R^n$. Finally $[\sigma:p]$ is clearly supported in $\overline\Om$ from its very definition.
\end{proof}
	
\subsection{Mixed boundary conditions}
	
When $\Gamma_D \neq \emptyset$ and $\Gamma_N\neq \emptyset$, the situation is much more delicate as in \cite{FG}. We first prove the following general result giving the required convexity inequality but only outside $\Sigma$ (see \cite[Theorem 6.2]{FG}) which, unfortunately, will not be enough for our purpose. We will later do additional assumptions in dimensions $n=2$ and $3$ which will ensure the validity of the convexity inequality in the whole $\R^n$.
	
\begin{proposition}
\label{prop:convexinequalityMixed}
Let $\Om \subset \R^n$ be a bounded open set with Lipschitz boundary. For every $\sigma \in \mathcal S_g \cap \mathcal K$ and $(u,e,p) \in \mathcal A_w$ with $H(p) \in \mathcal M(\Om \cup \Gamma_D)$, the restriction of the distribution $[\sigma:p]$ to $\R^n \setminus \Sigma$ extends to a bounded Radon measure in $\R^n \setminus \Sigma$ and
\begin{equation}
\label{eq:ar1}
H(p) \ge [ \sigma : p ] \quad \text{ in } \mathcal M(\R^n \setminus \Sigma ).	
\end{equation}	
\end{proposition}

\begin{proof}
Without loss of generality, we can assume $w=0$ in Definition \ref{definition:dualityMixed}. Let us fix a test function $\varphi \in C^\infty_c (\R^n \setminus \Sigma)$, and let $U \subset \R^n$ be an open set such that $\Sigma \subset U$ and $U \cap {\rm supp} (\varphi) = \emptyset$. Let us consider another open set $W \subset \R^n$ such that $\Gamma_N \setminus U \subset W$ and $\overline W \cap \partial \Om \subset \Gamma_N$. Finally, let $W'\subset \R^n$ be a further open set such that $W'\subset\subset W$, $\Gamma_N \setminus U \subset W'$ and ${\rm supp} (\varphi ) \cap \Gamma_N \subset W'$. Let $\psi \in C^\infty_c(\R^n)$ be a cut-off function such that $0 \le \psi \le 1$, ${\rm Supp}(\psi) \subset W$ and $\psi=1$ on $W'$. We decompose $\sigma$ 
as follows,
$$\sigma = \psi \sigma + (1 - \psi) \sigma =: \sigma_1 + \sigma_2.$$
Note that, for $i =1,2$, we have that $ \sigma_i \in H({\rm div}, \Om)$. Moreover,
\begin{equation}
\label{eq:extensionsigma1}
\sigma_1 \nu := \psi (\sigma  \nu) = \psi g \quad \textrm{on { $\partial\Om$} \quad and \quad  } {\sigma_2 = 0 \quad \text{on } W'. } 
\end{equation}
Substituting $\sigma$ with this decomposition in Definition \ref{definition:dualityMixed} we get that
\begin{align}\label{eq:sigma12} 
\langle [\sigma \colon p], \varphi \rangle &:= -\int_\Omega \varphi \sigma : e \,dx- \int_\Omega u \cdot {\rm div} \sigma \, \varphi \, dx - \int_\Omega \sigma \colon \big(u\odot \nabla \varphi\big) \, dx + \int_{ \Gamma_N}\varphi g \cdot u \, d \Hs^{n-1} \nonumber \\
& =  -\int_\Omega \varphi \sigma_1 : e \, dx- \int_\Omega u \cdot {\rm div} \sigma_1 \, \varphi \, dx - \int_\Omega \sigma_1 \colon \big(u\odot \nabla \varphi\big) \, dx +  \int_{ \Gamma_N}\varphi g \cdot u \, d \Hs^{n-1}\nonumber \\
& \quad  -\int_\Omega \varphi \sigma_2 : e \, dx- \int_\Omega u \cdot {\rm div} \sigma_2 \, \varphi \, dx - \int_\Omega \sigma_2 \colon \big(u\odot \nabla \varphi\big) \, dx.
\end{align}
		
We first approximate $(u,e,p)$ in the expression \eqref{eq:sigma12} involving $\sigma_1$. Indeed, thanks to  Lemma \ref{lem:approximationBDkin}, there exists a sequence $\{(u_k,e_k, p_k)\}_{k \in \N}$ in $C^{\infty} (\overline{\Om}; \R^n \times \M^n_\sym \times \M^n_\sym) $ such that $E u_k = e_k + p_k$ in $\Om$ and \eqref{eq:approximationuep}--\eqref{eq:Res} hold. On the one hand, we have
\begin{multline}\label{eq:divers-conv}
-\int_\Omega \varphi \sigma_1 : e_k \, dx- \int_\Omega u_k \cdot {\rm div} \sigma_1 \, \varphi \, dx - \int_\Omega \sigma_1 \colon \big(u_k\odot \nabla \varphi\big) \, dx + \int_{ \Gamma_N}\varphi g \cdot u_k \, d \Hs^{n-1} \\
\to -\int_\Omega \varphi \sigma_1 : e \, dx- \int_\Omega u \cdot {\rm div} \sigma_1 \, \varphi \, dx - \int_\Omega \sigma_1 \colon \big(u\odot \nabla \varphi\big) \, dx + \int_{ \Gamma_N}\varphi g \cdot u \, d \Hs^{n-1}.
\end{multline}		
On the other hand, for any $k \in \N$, thanks to the integration by parts formula for $\sigma_1 \in H(\Div,\Om)$ together with  \eqref{eq:extensionsigma1}, we can observe that
\begin{align}
& -\int_\Omega \varphi \sigma_1 : e_k \, dx- \int_\Omega u_k \cdot {\rm div} \sigma_1 \, \varphi \, dx - \int_\Omega \sigma_1 \colon \big(u_k\odot \nabla \varphi\big) \, dx +  \int_{ \Gamma_N}\varphi g \cdot u_k \, d \Hs^{n-1}\nonumber \\
&\qquad = \into{\varphi \sigma_1 : p_k} \, dx - \langle \sigma_1 \nu, \varphi u_k \rangle_{H^{-\frac{1}{2}} (\partial \Om;\R^n), H^{\frac{1}{2}} (\partial \Om;\R^n) }  + \int_{\Gamma_N}{\varphi g \cdot u_k \, d \Hs^{n-1} } \nonumber \\ 
& \qquad =  \into{\varphi \sigma_1 : p_k} \, dx -  \int_{\partial \Om}{\varphi { \psi g} \cdot u_k \, d \Hs^{n-1} }  + \int_{\Gamma_N}{\varphi g \cdot u_k \,d \Hs^{n-1} }\nonumber\\
& \qquad =  \into{\varphi \sigma_1 : p_k} \, dx,\label{eq:integrationbypartsrobin1}
\end{align}
where we used that $\psi=1$ on ${\rm Supp}(\varphi) \cap \Gamma_N$ and $\psi=0$ in $\partial\Om \setminus \Gamma_N$. Hence, by definition of the support function $H$, we have that $H(p_k) \geq \sigma:p_k$ a.e. in $\Om$. As a consequence,  if $\varphi \geq 0$,
\begin{eqnarray*}
\int_{\Om } H(p_k) \psi \varphi \, dx  &\geq  & \int_\Om \sigma_1:p_k \varphi \, dx\\
& =& -\into{\varphi u_k \cdot \Div \sigma_1 \, dx} - \into{\sigma_1 : (u_k \odot \nabla \varphi) \, dx} \\
&&\qquad- \into{\sigma_1 : e_k \varphi \,dx}   { + \int_{\Gamma_N}{\varphi g \cdot u_k \,d\Hs^{n-1}}}.
\end{eqnarray*}
We can  pass to the limit as $k \to \infty$ owing to \eqref{eq:Res} and \eqref{eq:divers-conv}. We deduce that
\begin{eqnarray}\label{eq:rob2}
\int_{W\cap \Om }  \varphi \, dH(p)  & \ge & \int_\Om \varphi\psi\, dH(p)\nonumber\\
& \geq & -\into{\varphi u \cdot {\rm div} \sigma_1 \, dx} - \into{\sigma_1 : (u \odot \nabla \varphi) \, dx}\nonumber\\
&&\qquad - \into{\sigma_1 : e \varphi \, dx} { + \int_{\Gamma_N}{\varphi g \cdot u \, d\Hs^{n-1}}}, 
\end{eqnarray}
where we have used the fact that $p$, hence $H(p)$, does not charge $\Gamma_N$.
		
\medskip
		
Coming back to \eqref{eq:sigma12}, we now approximate the last term in the right-hand side by approximating $\sigma_2$. Arguing as in \cite[Lemma 2.3]{DMDSM} or Step 2 in \cite[Theorem 6.2]{FG} and using \eqref{eq:extensionsigma1}, there exists a sequence $\lk \sigma^{k}_2 \rk_{k \in \N} \subset C^\infty(\overline \Om; \mathbb{M}^n_\sym)$ such that $\sigma_2^k (x) \in \mathbf K$ for all $x \in \overline\Om$ and
\begin{equation}
\label{eq:approximatiosigma2}
\begin{cases}
\sigma_2^k \to \sigma_2 & \textrm{strongly in }H({\rm div},\Om),\\
\sigma_2^k \nu = 0 & \textrm{on $W' \cap \Gamma_N$}.
\end{cases}
\end{equation}
Therefore, using the integration by parts formula in $BD(\Om)$, we infer that
\begin{equation}
\label{eq:approximationitgdirichlet}
\begin{split}
& -\int_\Omega \varphi \sigma_2^k : e \, dx- \int_\Omega u \cdot {\rm div} \sigma_2^k \, \varphi \, dx - \int_\Omega \sigma_2^k \colon \big(u\odot \nabla \varphi\big) \, dx \\
&\qquad = \int_{\Om}{\varphi \sigma_2^k :\, dp} - \int_{\partial \Om}{\varphi (\sigma_2^k \nu)\cdot u \,d\Hs^{n-1}} \\
&\qquad = \int_{\Om}{\varphi \sigma_2^k :\, dp} - \int_{ \Gamma_D}{\varphi (\sigma_2^k \nu)\cdot u \, d\Hs^{n-1}} \\
&\qquad=: \int_{\Om \cup \Gamma_D}{\varphi \sigma_2^k : \, dp}
\end{split}
\end{equation}
where in the second equality, we have used the fact that ${\rm supp} (\varphi) \cap \partial\Om \subset \Gamma_D \cup (\Gamma_N \cap W')$ and the last condition of \eqref{eq:approximatiosigma2}, while in the third equality we used that $p\mres \Gamma_D= - u \odot \nu \Hs^{n-1}\mres \Gamma_D$. Using that  $\sigma_2^k (x) \in \mathbf K$ for all $x \in \overline\Om$, we get that
$$\int_{\Om \cup \Gamma_D} \varphi \,dH(p) \geq \int_{\Om \cup \Gamma_D}{\varphi \sigma_2^k : \, dp},$$
hence passing to the limit as $k \to \infty$ using \eqref{eq:approximatiosigma2} and \eqref{eq:approximationitgdirichlet} leads to
\begin{equation}
\label{eq:rob1}
\int_{\Om \cup \Gamma_D} \varphi \,dH(p) \geq -\int_\Omega \varphi \sigma_2 : e \, dx- \int_\Omega u \cdot {\rm div} \sigma_2 \, \varphi \, dx - \int_\Omega \sigma_2 \colon \big(u\odot \nabla \varphi\big) \, dx.
\end{equation}
		
Combining \eqref{eq:sigma12}, \eqref{eq:rob2} and \eqref{eq:rob1}, we conclude that
$$\langle [\sigma:p],\varphi\rangle \leq \int_{\Om \cup \Gamma_D} \varphi\, dH(p)+\int_{W\cap \Om }  \varphi \, dH(p) .$$
Let us finally consider a decreasing sequence of open sets $\{W_j\}_{j \in \N}$ such that $\Gamma_N \setminus U \subset W_j$ and $W_j \cap \partial \Om \subset \Gamma_N$ for all $j \in \N$, and $\bigcap_j W_j=\overline{\Gamma_N \setminus U}$. Passing to the limit in the previous expression as $j \to \infty$ owing to the monotone convergence theorem yields
$$\langle [\sigma:p],\varphi\rangle \leq \int_{\Om \cup \Gamma_D} \varphi \,dH(p)+\int_{\overline{\Gamma_N \setminus U}}  \varphi \, dH(p) .$$
As $\overline{\Gamma_N \setminus U} \subset \Gamma_N \cup \Sigma$ and $p$ is concentrated on $\Om \cup \Gamma_D$, we deduce that 
$$\langle [\sigma:p],\varphi\rangle \leq \int_{\Om \cup \Gamma_D} \varphi \,dH(p)$$
which completes the proof of the proposition.
\end{proof}
		
In the remaining part of this section, we exhibit some particular cases where we can extend inequality  \eqref{eq:ar1} above into one in $\mathcal{M} (\R^n)$. The following result deals with the two-dimensional case where the convexity inequality holds provided $\Sigma$ is a finite set.
	
\begin{proposition}\label{prop:n=2}
Under the same assumptions as in Proposition \ref{prop:convexinequalityMixed}, assume further that $n=2$ and that $\Sigma$ is a finite set. Then, for all $\sigma \in \mathcal S_g \cap \mathcal K$ and all $(u,e,p) \in \mathcal A_w$,
$$H(p) \geq [\sigma:p] \quad \text{ in } \mathcal M(\R^2).$$
\end{proposition}
	
\begin{proof}
We again reduce to the case $w=0$. Arguing as in \cite[Example 2]{FG}, for all  $(u,e,p) \in \mathcal A_0$, there exists a sequence $\lk (u_k, e_k, p_k) \rk_{k \in \N}$ in $\mathcal A_0$ such that, for each $k \in \N$, $(u_k, e_k, p_k) =0 $  in an open neighborhood $U_k$ of $\Sigma$ and
\begin{equation}
\label{eq:weakapproximation}
\begin{cases}
u_k \to u \quad \textrm{strongly in $L^2(\Om; \R^2)$}, \\
e_k \to e \quad \textrm{strongly in $L^2(\Om; \mathbb{M}^2_\sym)$}, \\
p_k \rightharpoonup p \quad \textrm{weakly* in $\mathcal{M}(\Om\cup \Gamma_D; \mathbb{M}^2_\sym)$},\\
 |p_k|(\Om \cup \Gamma_D) \to |p|(\Om\cup \Gamma_D).
\end{cases}
\end{equation}
A careful inspection of the argument used in \cite[Example 2]{FG} shows that $|Eu_k|(\Omega) \to |Eu|(\Omega)$. Thus, applying \cite[Proposition 3.4]{Bab}, we deduce the convergence of the trace
\begin{equation}\label{eq:convtrace1}
u_k \to u\quad \text{ strongly in }L^1(\partial\Omega; \R^n).
\end{equation}
Moreover, for all $\varphi \in C^\infty_c(\R^2)$ with $\varphi \geq 0$,
\begin{equation}\label{eq:Res2}
\limsup_{k \to \infty} \int_{\Om \cup \Gamma_D} \varphi \, dH(p_k) \leq \int_{\Om \cup \Gamma_D} \varphi \, dH(p).
\end{equation}
Once more, \eqref{eq:Res2} does not follow from the Reshetnyak continuity Theorem because our $H$ does not fulfill the assumptions of that result.

Let $V_k$ be an open set satisfying $\Sigma \subset V_k \subset\subset U_k$, and let $\psi_k \in C^\infty_c(\R^2;[0,1])$ be a cut-off function such that $\psi_k=1$ in $V_k$ and ${\rm Supp}(\psi_k) \subset U_k$. For every $\varphi \in C^\infty_c(\R^2)$ with $\varphi\geq 0$, then $(1-\psi_k)\varphi \in C^\infty_c(\R^2 \setminus \Sigma)$ so that by Proposition \ref{prop:convexinequalityMixed},
$$\int_{\Om \cup \Gamma_D} \varphi \, dH(p_k) \geq \int_{\Om \cup \Gamma_D} \varphi(1-\psi_k) \, dH(p_k) \geq \langle [\sigma:p_k],\varphi(1-\psi_k)\rangle.$$
Since by construction ${\rm Supp}(u_k,e_k,p_k) \subset \R^2 \setminus U_k$, it is easily seen that ${\rm Supp}([\sigma:p_k]) \subset \R^2 \setminus U_k$ hence $ \langle [\sigma:p_k],\varphi\psi_k\rangle=0$. As a consequence
$$\int_{\Om \cup \Gamma_D} \varphi \, dH(p_k) \geq  \langle [\sigma:p_k],\varphi\rangle,$$
and the conclusion follows passing to the limit as $k \to \infty$ owing to the convergences \eqref{eq:weakapproximation}--\eqref{eq:Res2}.
\end{proof}
	
The three-dimensional case requires additional regularity assumptions for the domain $\Om$, and a particular geometric structure for the elasticity set $\mathbf K$ which has to be a cylinder whose axis is given by the set of spherical matrices. Note that these assumptions cover the physical cases of Von Mises and Tresca models.
	
\begin{proposition}\label{prop:n=3}
Under the same assumptions as in Proposition \ref{prop:convexinequalityMixed}, assume further that $n=3$ and that:
\begin{itemize}
\item[(i)] $\Om \subset \R^3$ is a bounded open set of class $C^2$ and $\Sigma$ is $1$-dimensional submanifold of class $C^2$;
\item[(ii)] $\mathbf K=K_D \oplus (\R\, {\rm Id})=\{\sigma  \in \mathbb M^3_{\rm sym} : \, \sigma_D\in K_D \}$ where $K_D \subset \mathbb M^3_D$ is a compact and convex set containing $0$ in its interior.
\end{itemize}
Then, for all $\sigma \in \mathcal S_g \cap \mathcal K$ and all $(u,e,p) \in \mathcal A_w$,
$$H(p) \geq [\sigma:p] \quad \text{ in } \mathcal M(\R^3).$$
\end{proposition}

\begin{proof}
Since $\sigma \in L^2(\Om; \mathbb{M}^3_\sym) $ satisfies $\Div\sigma \in L^2 (\Om; \R^3) $ and $\sigma_D \in L^\infty (\Om;\mathbb{M}^3_\sym )$ (because $\sigma \in \mathcal K$ implies $\sigma_D(x) \in K_D$ a.e. in $\Om$), we claim that $\sigma \in L^6 (\Om; \mathbb{M}^3_\sym)$. Indeed, arguing as in \cite[Proposition 6.1]{FG}, using the decomposition $\sigma = \sigma_D + \frac{1}{3} (\tr\sigma){\rm Id}$, we have that
$\frac13 \nabla (\tr\sigma)=\Div \sigma -\Div \sigma_D \in L^2(\Om;\R^3) + W^{-1,\infty}(\Om;\R^3)$, hence by the Sobolev embedding, 
$$\nabla(\tr\sigma) \in W^{-1,6} (\Om) + W^{-1, \infty} (\Om)  \subset W^{-1,6} (\Om).$$
Applying Ne\v cas Lemma (see \cite{Nec}), we infer that $\tr\sigma \in L^6(\Om)$, hence $\sigma \in L^6(\Om;\mathbb M^3_{\rm sym})$.
		
In particular, $\sigma \in  L^3(\Om;\mathbb M^3_{\rm sym})$, $\sigma_D \in L^\infty(\Om;\mathbb M^3_D)$, $\Div \sigma \in L^{3/2}(\Om;\R^3)$ and $\sigma\nu\in L^\infty(\Gamma_N;\R^3)$. These conditions turn out to be sufficient to apply \cite[Proposition 2.7]{KT} (with, in the notation of \cite{KT}, $n=3$, $p=3/2$ and $p^*=3$). Then, an immediate adaptation of the proof of \cite[Lemma 3.5]{KT} (using \cite[Proposition 2.7]{KT} instead of \cite[Corollary 2.8]{KT}) shows the validity of the so-called Kohn-Temam condition:
$$\lim_{\delta \to 0} \frac{1}{\delta}\int_{\Sigma_\delta} |\sigma| |u|\, dx =0,$$
where $\Sigma_\delta:=\Om \cap \{x \in \R^3 : \, {\rm dist}(x,\Sigma)<\delta\}$. We are thus in position to argue as in the proof of \cite[Theorem 6.5]{FG} to get the conclusion. Indeed, let $\psi_\delta \in C^\infty_c(\Sigma_\delta;[0,1])$ be a cut-off function such that $\psi_\delta=1$ in a neighborhood of $\Sigma$ and $|\nabla \psi_\delta|\leq 2/\delta$. Then, for all $\varphi \in C^\infty_c(\R^3)$ with $\varphi\geq 0$, we have 
\begin{multline*}
\langle[\sigma:p],(1-\psi_\delta)\varphi\rangle=-\int_\Omega (1-\psi_\delta)\varphi \sigma : e \,dx- \int_\Omega u \cdot {\rm div} \sigma \, (1-\psi_\delta)\varphi \, dx - \int_\Omega (1-\psi_\delta) \sigma \colon \big(u\odot \nabla \varphi\big) \, dx\\
+ \int_\Omega \varphi \sigma \colon \big(u\odot \nabla \psi_\delta\big) \, dx + \int_{ \Gamma_N}(1-\psi_\delta)\varphi g \cdot u \, d \Hs^{n-1}.
\end{multline*}
Since $\psi_\delta \searrow 0$ pointwise, and
$$\left|  \int_\Omega \varphi \sigma \colon \big(u\odot \nabla \psi_\delta\big) \, dx\right| \leq  \frac{2\|\varphi\|_{L^\infty(\Om)}}{\delta}\int_{\Sigma_\delta} |\sigma| |u|\, dx \to 0,$$
the dominated convergence Theorem allows us to pass to the limit as $\delta \to 0$, and get that
$$\langle[\sigma:p],(1-\psi_\delta)\varphi\rangle \to\langle[\sigma:p],\varphi\rangle.$$
On the other hand, since $(1-\psi_\delta)\varphi \in C^\infty_c(\R^3 \setminus \Sigma)$, Proposition \ref{prop:convexinequalityMixed} ensures that
$$\int_{\Om \cup \Gamma_D} \varphi \, dH(p) \geq \int_{\Om \cup \Gamma_D} (1-\psi_\delta)\varphi \, dH(p) \geq \langle[\sigma:p],(1-\psi_\delta)\varphi\rangle.$$
The conclusion follows passing to the limit as $\delta \to 0$.
\end{proof}
	
\section{Dynamic elasto-plasticity}
	
\subsection{The model with dissipative boundary conditions}
	
We consider a small strain dynamical perfect plasticity problem under the following assumptions:
	
\noindent $(H_3)$ {\bf The elastic properties.} We assume that the material is isotropic, which means that the constitutive law, expressed by Hooke's tensor, is given by
$${\bf A} \xi = \lbd ({\rm tr}\,\xi) {\rm Id} + 2 \mu \xi \quad {\textrm{for all } } \xi \in \mathbb{M}^n_{{\rm sym}},$$ 
where $\lbd$ and $ \mu$ are the Lam\'e coefficients satisfying $\mu > 0$ and $2 \mu + n  \lbd >0$. These conditions imply the existence of constants $\alpha>0$ and $\beta>0$ such that
$$\alpha |\xi|^2 \leq \mathbf A \xi:\xi \leq \beta |\xi|^2 \quad \text{ for all }\xi \in \mathbb{M}^n_{{\rm sym}}.$$
	
We define the following quadratic form
$$Q(\xi) := \frac{1}{2 }\mathbf A \xi:\xi = \frac{\lbd}{2 }  ({\rm tr}\,\xi)^2 + \mu \abs{\xi}^2 \quad {\textrm{for all } } \xi \in \mathbb{M}^n_{{\rm sym}}.  $$
If $e \in L^2 (\Om; \mathbb{M}^n_{{\rm sym}})$, we further define the elastic energy by
$$\Q (e) := \into{Q(e) \, dx}.$$
	
$(H_4)$ {\bf The dissipative boundary conditions.} Let $S \in L^\infty(\partial\Om;\mathbb M^n_{\rm sym})$ be a boundary matrix satisfying the conditions: there exists a constant $c>0$ such that
$$S(x)z\cdot z \geq c|z|^2 \quad \text{for $\mathcal H^{n-1}$-a.e. $x \in \partial\Om$} \text{ and for all }z \in \R^n.$$
	
$(H_5)$ {\bf The external forces.}  We assume the body is subjected to external body forces
$$f \in H^1(0,T; L^2(\Om; \R^n)).$$
	
$(H_6)$ {\bf The initial conditions.} Let $u_0 \in H^1(\Om; \R^n)$, $v_0 \in H^2(\Om; \R^n)$, $e_0 \in L^2(\Om;  \mathbb{M}^n_{{\rm sym}})$ and $p_0 \in  L^2(\Om;  \mathbb{M}^n_{{\rm sym}})$ be such that
\begin{equation*}
\begin{cases}
\sigma_0 := {\bf A} e_0 \in \mathcal K,\\
Eu_0 = e_0 + p_0 & \textrm{ in } \Om, \\
S v_0  + \sigma_0 \nu =0 & \textrm{ on } \partial \Om.
\end{cases}
\end{equation*}

In order to formulate the main result of \cite{BC}, we further need to introduce	the function $\psi:\partial\Om \times \R^n \to \R^+$ defined by
\begin{equation}\label{eq:psi}
\psi (x,z) = \inf_{w \in \R^n} \left\{\frac12 S(x)w\cdot w + H((w-z)\odot \nu(x))\right\} \text{ for $\mathcal H^{n-1}$-a.e. $x \in \partial\Om$ and all $z \in \R^n$,}
\end{equation}	
where $\nu(x)$ is the outer normal to $\Om$ at $x \in \partial\Om$. We recall  (see \cite[Remark 4.7]{BC}) that the differential of $\psi$ in the $z$-direction is given by
$$D_z \psi (x,z) = {\rm{P}}_{-{\mathbf K} \nu(x) } (S (x) z),$$  
where ${\rm P}_{-\mathbf K\nu(x)}$ is the orthogonal projection in $\R^n$ onto the closed and convex set $-\mathbf K\nu(x)$ with respect to the scalar product $(u,v) \in \R^n \times \R^n \mapsto \langle u,v\rangle_{S(x)^{-1}}:= S(x)^{-1} u \cdot v$. We further denote by $\| \cdot\|_{S(x)^{-1}}$ its associated norm. 

The following well posedness result with homogeneous dissipative boundary conditions has been established in \cite{BC}. 	
\begin{theorem}
\label{theorem:wellposednesslevelS}
Assume that assumptions $(H_1)$--$(H_6)$ hold. Then, there exists a unique triple $(u,e,p)$ such that
\begin{equation*}
\begin{cases}
u \in W^{2, \infty} (0,T; L^2 (\Om; \R^n)) \cap C^{0,1} ( \left[0, T\right]; BD (\Om)), \\
e \in W^{1, \infty} (0,T; L^2 (\Om ; \mathbb{M}^n_{{\rm sym}} )), \\
p \in C^{0,1} ( \left[0, T\right]; \mathcal{M} (\Om;\mathbb{M}^n_{{\rm sym}})) , \\
\end{cases}
\end{equation*}
\begin{equation*}
\sigma:= \mathbf{A} e \in L^\infty  (0,T; H( {\rm div}, \Om)), \quad \sigma \nu \in L^\infty (0,T; L^2 (\partial \Om; \R^n)),
\end{equation*}
and satisfying
\begin{enumerate}
\item The initial conditions:
$$u(0) = u_0, \quad \dot{u}(0) = v_0, \quad e(0)=e_0, \quad p(0) = p_0;$$
\item The additive decomposition: for all $t \in \left[0,T \right]$,
$$E u(t) = e(t) + p(t) \quad \textrm{in }\mathcal{M} (\Om;\mathbb{M}^n_{{\rm sym}});$$
\item The equation of motion:
$$\ddot {u} - \Div \sigma = f \quad \textrm{in }L^2 (0,T; L^2(\Om ; \R^n)); $$
\item The relaxed dissipative boundary condition:
$${\rm P}_{-\mathbf{K} \nu} (S \dot{u}) + \sigma \nu = 0 \quad \textrm{in }  L^2 (0,T; L^2 (\partial \Om; \R^n));$$
\item The stress constraint: for every $t \in \left[ 0, T \right]$,
$$\sigma (t) \in \mathbf{K} \quad \text{ a.e. in }\Om;$$
\item The flow rule: for a.e. $t \in \left[0, T \right]$,
$$H (\dot {p} (t) )= \left[ \sigma(t): \dot{p}(t) \right] \quad \textrm{in } \mathcal{M} (\Om);$$
\item The energy balance: for every $t \in \left[0, T \right]$
\begin{align}
&\frac{1}{2} \int_{\Om}{\abs{\dot{u} (t)}^2 \, dx} + \mathcal{Q}(e(t)) + \int_0^t  H(\dot{p}(s))(\Om)\, ds +   \intzt{\int_{\partial \Om}{\psi(x, \dot{u})\, d\Hs^{n-1}}\, ds} \nonumber \\ 
& \quad + \frac{1}{2} \intzt{\intpo{ S^{-1} (\sigma \nu) \cdot (\sigma \nu) \, d \Hs^{n-1}}\, ds}= \frac{1}{2} \int_{\Om}{\abs{v_0}^2 \, dx} + \mathcal{Q}(e_0) + \intzt{\into{f \cdot \dot{u}\, dx}\, ds}. 
\label{eq:energybalance}
\end{align}
\end{enumerate}
Moreover, the following uniform estimate holds
\begin{equation}
\norm{\ddot{u}}^2_{L^\infty (0,T; L^2 (\Om; \R^n))} + \norm{\dot{e}}^2_{L^\infty(0,T; L^2 (\Om; \mathbb{M}^n_{{\rm sym}}))}  \le C_* ,
\label{posteriori}
\end{equation}
for some constant $C_*>0$ depending on $\|u_0\|_{H^1(\Om;\R^n)}$, $\|v_0\|_{H^2(\Om;\R^n)}$, $\|e_0\|_{L^2(\Om;\mathbb M^n_{\rm sym})}$, $\|\sigma_0\|_{H({\rm div},\Om)}$ and $\|p_0\|_{L^2(\Om;\mathbb M^n_{\rm sym})}$, but which is independent of $S$.
\end{theorem}
	
\subsection{Derivation of mixed boundary condition}
	
Our aim is to show through an asymptotic analysis how it is possible to obtain homogeneous mixed boundary conditions starting from dissipative boundary conditions. We consider a boundary matrix of the form 
$$S(x)=S_\lambda(x):= \left( \lambda { {\bf 1}_{\Gamma_D}(x) + \frac{1}{\lbd} {\bf 1}_{\Gamma_N}(x)} \right){\rm Id}, \quad \lambda>0.$$		
\begin{remark}
{\rm Note that since 
$$\| \cdot \|_{S_\lambda(x)^{-1}}=\left( \lambda { {\bf 1}_{\Gamma_D} (x) + \frac{1}{\lbd} {\bf 1}_{\Gamma_N}(x) } \right)^{-1} | \cdot |,$$
for any $\lbd >0$ and all $x \in \partial \Om \setminus \Sigma$, the orthogonal projection ${\rm P}_{-\mathbf K\nu(x)}$ onto the closed and convex set $-\mathbf K\nu(x)$  with respect to the scalar product $\langle\cdot,\cdot\rangle_{S_\lambda(x)^{-1}}$ coincides with the orthogonal projection with respect to the canonical Euclidean scalar product of $\R^n$. It is in particular independent of $\lambda$. }
\end{remark}
	
We will need to strengthen assumption $(H_1)$ into
	
\noindent $(H'_1)$ {\bf Reference configuration.} Let $\Om \subset \R^n$ be a bounded open set with $C^3$ boundary. We assume that $\partial \Om$ is decomposed as the following disjoint union
$$\partial \Om = \Gamma_D \cup \Gamma_N \cup \Sigma,$$
where $\Gamma_D$ and $\Gamma_N$ are open sets in the relative topology of $\partial\Om$, and $\Sigma = \partial _{| \partial \Om}\Gamma_D = \partial _{| \partial \Om}\Gamma_N$ is a $(n-2)$-dimensional submanifold of class $C^3$. 
	
Moreover, the initial condition needs to be adapted to our mixed boundary conditions.
	
\noindent $(H'_6)$ {\bf The initial conditions.} Let $u_0 \in H^1(\Om; \R^n)$, $v_0 \in H^2(\Om; \R^n) $, $e_0 \in L^2(\Om;  \mathbb{M}^n_{{\rm sym}})$, $p_0 \in  L^2(\Om;  \mathbb{M}^n_{{\rm sym}})$ and $\sigma_0 := {\bf A} e_0 \in H^2(\Om;\M^n_\sym)$ be such that
$$\begin{cases}
Eu_0 = e_0 + p_0& \text{ in }\Om,\\
v_0 = 0& \text{ on }\Gamma_D,\\
\sigma_0 \nu =0 & \text{ on }\Gamma_N,\\
\sigma_0 + B(0,r) \subset {\bf K} & \text{ in }\Om \text{ for some }r>0.
\end{cases}$$
		
First, we are going to construct a sequence of initial data $(u^\lambda_0, v^\lambda_0, e^\lambda_0, p^\lambda_0)$ satisfying $(H_6)$ with $S=S_\lambda$, and approximating $(u_0,v_0,e_0,p_0)$ as $\lambda \to \infty$. This is the object of the following result.

\begin{lemma}
\label{lemma:6.1}
Let $n=2$, $3$. Under assumptions $(H'_1)$ and $(H'_6)$, for every $\lambda>0$, there exists $(v_0^\lbd, \sigma_0^\lbd) \in H^2(\Om;\R^n) \times \mathcal K $ such that $(v_0^\lbd ,\sigma_0^\lbd) \to (v_0,\sigma_0)$ strongly in $  H^2(\Om;\R^n) \times H(\Div,\Om)$ as $\lbd \to \infty$ and
\begin{equation}
\label{approximate:initialdata}
\left(  \lbd {\bf 1}_{\Gamma_D}  + \frac 1\lbd {\bf 1}_{\Gamma_N} \right) v_0^\lbd + \sigma_0^\lbd \nu = 0 \quad \mathcal H^{n-1}\textrm{-a.e. on $\partial \Om$}.
\end{equation}   
\end{lemma}
\begin{proof}
Since $\partial \Om$ has a $C^3$ boundary then its normal $\nu$ belongs to $C^2(\partial\Om;\R^n)$ and, thanks to the Trace Theorem in Sobolev spaces, the trace of $\sigma_0$ belongs to $H^\frac{3}{2}(\partial \Om;\mathbb M^n_{\rm sym})$. As a consequence, the product $\sigma_0 \nu$ belongs to $H^{\frac32}(\partial\Om;\R^n)$ and there exists an extension $\hat{v}_0 \in H^2(\Om; \R^n)$ whose trace on $\partial\Om$ coincides with $-\sigma_0 \nu$ with the estimate 
$$\|\hat v_0\|_{H^2(\Om;\R^n)} \leq C  \|\sigma_0\nu\|_{H^{3/2}(\partial\Om;\R^n)},$$
where $C>0$ is a constant only depending on $n$ and $\Om$. For each $\lbd>0$, let us define 
$$v_0^\lbd := v_0 + \lambda^{-1}\hat{v}_0  \in H^2(\Om; \R^n).$$
It follows 
that $v_0^\lbd \raw v_0$ strongly in $H^2(\Om; \R^n)$ as $\lbd \to  \infty$. Now, we consider $z_0 \in H^1(\Om;\R^n)$ as the unique weak solution of the boundary value problem
\begin{equation}
\label{eqlem:5.1}
\begin{cases}
z_0-{{\rm div}} (e(z_0)) =0 & {{\rm in}} \ \Om ,\\
e(z_0) \nu = - v_0 & {{\rm on}} \ \partial \Om.
\end{cases}
\end{equation}
According to Korn's inequality and the Lax-Milgram Lemma such a solution exists and is unique. Using that $\Om$ has a $C^3$-boundary and that $v_0 \in H^{\frac32}(\partial\Om;\R^n)$, elliptic regularity ensures that $z_0 \in H^3 (\Om; \R^n)$. Let us define
$$\sigma_ 0^\lbd := \sigma_0+\lbd^{-1} e(z_0)$$
In particular, $ \sigma_0^\lambda \to \sigma_0$ strongly in $H({\rm div},\Om)$ as $\lbd \to \infty$. On $\Gamma_D$, we observe that 
$$\lbd {v_0^\lbd}_{| \Gamma_D}   +{\sigma_0^\lbd \nu}_{| \Gamma_D} = \lbd {v_0}_{| \Gamma_D} + \hat{v_0}_{| \Gamma_D} + \sigma_0\nu_{| \Gamma_D}+ \frac{1}{\lbd} {e(z_0) \nu}_{| \Gamma_D} = 0,$$
where we have used the fact that $e(z_0) \nu = -v_0 = 0$ and $\hat v_0=-\sigma_0\nu$ on $\Gamma_D$.
Similarly, on $\Gamma_N$ we have
$$\frac{1}{\lbd} {v_0^\lbd}_{| \Gamma_N}   +{\sigma_0^\lbd \nu}_{| \Gamma_N} = \frac{1}{\lambda} {v_0}_{| \Gamma_N} +\frac{1}{\lbd^2} \hat{v_0}_{| \Gamma_N} + \sigma_0\nu_{| \Gamma_N}+ \frac{1}{\lbd} {e(z_0) \nu}_{| \Gamma_N} = 0,$$
where we have used the fact that $\hat v_0=-\sigma_0\nu=0$ and $e(z_0)\nu = -v_0$ on $\Gamma_N$. We conclude \eqref{approximate:initialdata} thanks to the fact that $\partial \Om = \Gamma_D \cup \Gamma_N \cup \Sigma$ and $\Hs^{n-1}(\Sigma)=0$. 
		
It remains to check that $\sigma_0^\lambda \in \mathbf K$ a.e. in $\Om$.  To this aim, we have by Sobolev imbedding (recall that $n=2$ or $3$) that $e(z_0) \in H^2(\Om;\M^n_\sym) \subset L^\infty(\Om;\M^n_\sym)$. Let $r>0$ be the constant given by the last property of hypothesis $(H'_6)$ and $\lambda>0$ large enough so that $\lambda^{-1} \|e(z_0)\|_{L^\infty(\Om;\M^n_\sym)} < r$. It thus follows that $\sigma_0^\lambda \in \sigma_0 + B(0,r) \subset \mathbf K$ a.e. in $\Om$.
\end{proof}
	
Given the initial data $(u_0,v_0^\lambda,e_0^\lambda:=\mathbf A^{-1}\sigma^\lambda_0,p_0^\lambda:=Eu_0-\mathbf A^{-1}\sigma_0^\lambda)$ satisfying $(H_6)$, we denote by $(u_\lambda,e_\lambda,p_\lambda)$ the associated solution given by Theorem \ref{theorem:wellposednesslevelS}. Our aim is to study the asymptotic behavior of the solutions $(u_\lambda,e_\lambda,p_\lambda)$ when $\lbd \to  \infty$ in order to recover Dirichlet ($\Gamma_N = \emptyset$), Neumann ($\Gamma_D = \emptyset$) {and mixed} boundary conditions in the other cases.
	
Our main result is the following:
\begin{theorem}
\label{thm:compactness}
Assume that $(H'_1)$, $(H_2)$, $(H_3)$, $(H_5)$ and  $(H'_6)$ hold. For each $\lambda>0$, let $(v_0^\lambda,\sigma_0^\lambda)$ be given by Lemma \ref{lemma:6.1}, and let $(u_\lambda,e_\lambda,p_\lambda)$ be the solution given by Theorem \ref{theorem:wellposednesslevelS} associated with the boundary matrix $S_\lambda$ defined in \eqref{eq:Slambda} and the initial data $(u_0,v_0^\lambda,e_0^\lambda:=\mathbf A^{-1}\sigma^\lambda_0,p_0^\lambda:=Eu_0-\mathbf A^{-1}\sigma_0^\lambda)$. Then,
$$\begin{cases}
u_\lambda \rightharpoonup u & \text{ weakly* in } W^{2,\infty}(0,T;L^2(\Om;\R^n)),\\
e_\lambda \rightharpoonup e & \text{ weakly* in } W^{1,\infty}(0,T;L^2(\Om;\mathbb M^n_{\rm sym})),\\
\sigma_\lambda \rightharpoonup \sigma & \text{ weakly* in } W^{1,\infty}(0,T;L^2(\Om;\mathbb M^n_{\rm sym})),\\
p_\lambda(t) \rightharpoonup p(t) & \text{ weakly* in }\mathcal M(\Om;\mathbb M^n_{\rm sym}) \text{ for all }t \in [0,T],
\end{cases}
$$
where $(u,e, p)$ is the unique triple satisfying 
$$
\begin{cases}
u \in W^{2,\infty}(0,T;L^2(\Omega;\R^n)) \cap C^{0,1}([0,T];BD(\Omega)),\\
e \in W^{1,\infty}(0,T;L^2(\Om;\mathbb M^n_{\rm sym})),\\
\sigma:=\mathbf A e \in W^{1,\infty}(0,T;L^2(\Om;\mathbb M^n_{\rm sym})) \cap L^\infty(0,T;H({\rm div},\Omega)),\\
p\in C^{0,1}([0,T];\mathcal M(\Om \cup \Gamma_D;\mathbb M^n_{\rm sym})),
\end{cases}
$$
together with
\begin{enumerate}
\item The initial conditions:
$$u(0) = u_0, \quad  \dot{u} (0) = v_0,\quad  e (0) = e_0, \quad p (0) = p_0; $$
\item The kinematic compatibility: for all $t \in \left[0, T \right]$,
$$\begin{cases}
E u(t) = e(t) + p(t) & \textrm{in $\Om$},\\
p(t) =- u(t)  \odot \nu \Hs^{n-1}  & \textrm{on $\Gamma_D$};
\end{cases}$$
			
\item The equation of motion:
$$\ddot{u} - {\rm{div}}\sigma =f \quad in \ L^2 (0,T;L^2(\Om; \R^n)); $$
\item The stress constraint: for every $t \in \left[0, T \right]$,
$$\sigma(t) \in \mathbf{K} \quad {\rm{a.e \ in \ }} \Om; $$
\item The boundary condition
$$ \sigma \nu = 0 \quad {\rm in } \quad L^2(0,T; L^2(\Gamma_N; \R^n)); $$
\item The flow rule: if one of the following conditions are satisfied:
\begin{itemize}
\item[(i)] Dirichlet case: $\Om=\Gamma_D$,
\item[(ii)] Neumann case: $\Om=\Gamma_N$,
\item[(iii)] Mixed case in dimension $n=2$: $\Gamma_D \neq \emptyset$, $\Gamma_N\neq \emptyset$ and $\Sigma$ finite,
\item[(iv)] Mixed case in dimension $n=3$: $\Gamma_D \neq \emptyset$, $\Gamma_N\neq \emptyset$ and 
$$\mathbf K=K_D \oplus (\R {\rm Id}):=\{\sigma  \in \mathbb M^3_{\rm sym} : \, \sigma_D \in K_D\},$$
for some compact and convex set $K_D \subset \mathbb M^3_D$ containing zero in its interior, 
\end{itemize}
then, for a.e. $t \in [0,T]$,
$$H (\dot{p}(t)) = [ \sigma(t): \dot{p}(t) ] \quad {\rm{in \ }} \mathcal{M}(\Om \cup \Gamma_D).$$
\end{enumerate} 
\end{theorem} 
	
As explained before, the solution $(u,e,p)$ to the previous boundary value problem will be obtained by means of an asymptotic analysis as $\lbd \to \infty$ of the solution  $(u_\lambda,e_\lambda,p_\lambda)$ of the dissipative boundary value in the Theorem \ref{theorem:wellposednesslevelS}. This analysis is based in the spirit of \cite[Theorem 5.1]{BM} in the antiplane case.
	
\subsection{Weak compactness and passing to the limit into linear equations}
	
We observe that the constant $C_*>0$ appearing in estimate \eqref{posteriori} of Theorem \ref{theorem:wellposednesslevelS} depends on the various norms $ \|u_0\|_{H^1(\Om;\R^n)}$, $\|v_0^\lambda\|_{H^2(\Om;\R^n)}$, $\|e_0^\lambda\|_{L^2(\Om;\mathbb M^n_{\rm sym})}$, $\|\sigma^\lambda_0\|_{H({\rm div},\Om)}$ and $\|p_0^\lambda\|_{L^2(\Om;\mathbb M^n_{\rm sym})}$ of the initial data. Since, by Lemma \ref{lemma:6.1}, these quantities are independent of $\lambda$, it follows that the constant $C_*$ is independent of $\lambda$ as well. This is essential to get uniform bounds on the sequence $\{(u_\lambda,e_\lambda,p_\lambda)\}_{\lambda>0}$ and then weak compactness thereof.
	
The following compactness result follows from standard argument as, e.g., in \cite[Section 5]{BM}. The weak convergences allow us to obtain, in the limit, the initial conditions, the kinetic compatibility, the equation of motion and the stress constraint, 
	
\begin{lemma}\label{lem:compactness}
Assume that $(H'_1)$, $(H_2)$, $(H_3)$, $(H_5)$ and  $(H'_6)$ hold. There exist a subsequence (not relabeled) and 
$$
\begin{cases}
u \in W^{2,\infty}(0,T;L^2(\Omega;\R^n)) \cap C^{0,1}([0,T];BD(\Omega)),\\
e \in W^{1,\infty}(0,T;L^2(\Om;\mathbb M^n_{\rm sym})),\\
\sigma\in W^{1,\infty}(0,T;L^2(\Om;\mathbb M^n_{\rm sym})) \cap L^\infty(0,T;H({\rm div},\Omega)),\\
p \in C^{0,1}([0,T];\mathcal M(\Om;\mathbb M^n_{\rm sym})),
\end{cases}
$$
such that as $\lambda \to \infty$,
$$
\begin{cases}
u_\lambda \rightharpoonup u & \text{ weakly* in } W^{2,\infty}(0,T;L^2(\Om;\R^n)),\\
e_\lambda \rightharpoonup e & \text{ weakly* in } W^{1,\infty}(0,T;L^2(\Om;\mathbb M^n_{\rm sym})),\\
\sigma_\lambda \rightharpoonup \sigma & \text{ weakly* in } W^{1,\infty}(0,T;L^2(\Om;\mathbb M^n_{\rm sym})),
\end{cases}
$$
and, for every $t \in [0,T]$,
$$
\begin{cases}
u_\lambda(t) \rightharpoonup u(t) & \text{ weakly in } L^2(\Om;\R^n),\\
u_\lambda(t) \rightharpoonup u(t) & \text{ weakly* in } BD(\Om),\\
\dot u_\lambda(t) \rightharpoonup \dot u(t) & \text{ weakly in } L^2(\Om;\R^n),\\
e_\lambda(t) \rightharpoonup e(t) & \text{ weakly in } L^2(\Om;\mathbb M^n_{\rm sym}),\\
\sigma_\lambda(t) \rightharpoonup \sigma(t) & \text{ weakly in } L^2(\Om;\mathbb M^n_{\rm sym}),\\
p_\lambda(t) \rightharpoonup p(t) & \text{ weakly* in }\mathcal M(\Om;\mathbb M^n_{\rm sym}).
\end{cases}
$$
Moreover, there hold: 
\begin{itemize}
\item the initial conditions: $ u(0) = u_0, \  \dot{u} (0) = v_0,\  e (0) = e_0, \ p (0) = p_0$;
\item the additive decomposition: for all $t \in [0, T]$,
$$E u(t) = e(t) + p(t) \quad \text{ in } \mathcal M(\Om;\mathbb M^n_{\rm sym});$$
\item the equation of motion: $\ddot{u} - {\rm{div}}\sigma=f$  in $L^2 (0,T;L^2(\Om; \R^n))$;
\item the stress constraint: for every $t \in \left[0, T \right]$,
$\sigma (t) =\mathbf A e(t) \in \mathbf{K}$ a.e  in $\Om$;
\item the Neumann condition: $\sigma\nu=0$ in $L^2(0,T;L^2(\Gamma_N;•\R^n))$.
\end{itemize}
\end{lemma}
	
\begin{proof}
According to the energy balance \eqref{eq:energybalance} and estimate \eqref{posteriori}, we infer that
\begin{multline}
\|\dot{u}_\lambda\|_{L^\infty(0,T;L^2(\Om;\R^n))} +\|\sigma_\lambda\|_{L^\infty(0,T;L^2(\Om;\mathbb M^n_{\rm sym}))} + \|\dot p_\lambda\|_{L^1(0,T;\mathcal M(\Om;\mathbb M^n_{\rm sym}))}\\
+ \frac{1}{\sqrt\lbd} \|\sigma_\lambda \nu\|_{L^2(0,T;L^2({ \Gamma_D};\R^n))} { + \sqrt{\lbd} \|\sigma_\lambda \nu\|_{L^2(0,T;L^2({\Gamma_N};\R^n))} }\\
+  \int_0^T\int_{\partial \Om}{\psi_\lbd(x, \dot{u}_\lbd)\, d\Hs^{n-1}}\, ds \leq C, \label{eq:differenceenergies2}
\end{multline}
where $\psi_\lambda$ is given by \eqref{eq:psi} with $S=S_\lambda$, and
$$\norm{\ddot{u}_\lambda}^2_{L^\infty (0,T; L^2 (\Om; \R^n))} + \norm{\dot{e}_\lambda}^2_{L^\infty(0,T; L^2 (\Om; \mathbb{M}^n_{{\rm sym}}))}  \le C_*.$$
In both previous estimates, the constants $C>0$ and $C_*>0$ are independent of $\lbd$. Using that $u_\lambda \in W^{2,\infty}(0,T;L^2(\Om;\R^n))$ and $u_0\in L^2(\Om;\R^n)$, we get
$$\sup_{\lambda>0}\|u_\lambda\|_{W^{2,\infty}(0,T;L^2(\Om;\R^n))}<\infty,$$
and similarly, since $e_\lambda \in W^{1,\infty}(0,T;L^2(\Om;\mathbb M^n_{\rm sym}))$ and $e_0 \in L^2(\Om;\mathbb M^n_{\rm sym})$, 
$$\sup_{\lambda>0}\|e_\lambda\|_{W^{1,\infty}(0,T;L^2(\Om;\mathbb M^n_{\rm sym}))}<\infty.$$
We can thus extract a subsequence (not relabeled) and find $u \in W^{2,\infty}(0,T;L^2(\Om;\R^n))$ and $e\in W^{1,\infty}(0,T;L^2(\Om;\mathbb M^n_{\rm sym}))$ such that, as $\lambda \to \infty$,
$$
\begin{cases}
u_\lambda \rightharpoonup u\quad \text{ weakly* in } W^{2,\infty}(0,T;L^2(\Om;\R^n)),\\
e_\lambda \rightharpoonup e \quad \text{ weakly* in } W^{1,\infty}(0,T;L^2(\Om;\mathbb M^n_{\rm sym})).
\end{cases}
$$
Setting $\sigma:=\mathbf Ae\in W^{1,\infty}(0,T;L^2(\Om;\mathbb M^n_{\rm sym}))$ we also have
$$\sigma_\lambda \rightharpoonup \sigma\quad \text{ weakly* in } W^{1,\infty}(0,T;L^2(\Om;\mathbb M^n_{\rm sym})),$$
and using the equation of motion leads to 
$${\rm div}\sigma_\lambda=\ddot u_\lambda-f \rightharpoonup \ddot u-f   \quad \text{ weakly* in  }L^\infty(0,T;L^2(\Om;\R^n)).$$
By uniqueness of the distributional limit, we infer that $\Div \sigma=\ddot u-f  \in L^\infty(0,T;L^2(\Om;\R^n))$ and, thus, $\sigma \in L^\infty(0,T;H({\rm div},\Om))$.
		
Owing to Ascoli-Arzela Theorem, for every $t \in [0,T]$,
$$
\begin{cases}
u_\lambda(t) \rightharpoonup u(t) & \text{ weakly in } L^2(\Om;\R^n),\\
\dot u_\lambda(t) \rightharpoonup \dot u(t) & \text{ weakly in } L^2(\Om;\R^n),\\
e_\lambda(t) \rightharpoonup e(t) & \text{ weakly in } L^2(\Om;\mathbb M^n_{\rm sym}),\\
\sigma_\lambda(t) \rightharpoonup \sigma(t) & \text{ weakly in } L^2(\Om;\mathbb M^n_{\rm sym}).
\end{cases}
$$
We now derive weak compactness on the sequence $\{p_\lambda\}_{\lambda>0}$ of plastic strains. Thanks to the energy balance between two arbitrary times $0 \leq t_1 \leq t_2 \leq T$ together with \eqref{eq:coercH},
\begin{align}
r\int_{t_1}^{t_2}|\dot p_\lambda(s)|(\Om)\, ds \leq  \intot{H(\dot{p}_\lbd (s))(\Om)\, ds} & \leq & \frac{1}{2} \int_{\Om}(\dot u_\lambda(t_1)-\dot u_\lambda(t_2))\cdot (\dot u_\lambda(t_1)+\dot u_\lambda(t_2))\, dx \nonumber\\
&&+ \frac{1}{2} \int_{\Om}(\sigma_\lambda(t_1)-\sigma_\lambda(t_2)): (e_\lambda(t_1)+e_\lambda(t_2))\,dx\nonumber\\
&& +\int_{t_1}^{t_2}\int_\Om f\cdot \dot u_\lambda\, dx \, ds. \label{eq:differenceenergies}
\end{align}
By $(H_5)$, using that $f \in L^\infty(0,T;L^2(\Om;\R^n))$, that $\{\dot u_\lambda\}_{\lambda>0}$ is bounded in $L^\infty(0,T;L^2(\Om;\R^n))$ and that $\{\sigma_\lambda\}_{\lambda>0}$ is bounded in $L^\infty(0,T;L^2(\Om;\mathbb M^n_{\rm sym}))$, we can find a constant $C>0$ independent of $\lambda$ such that
$$|p_\lambda(t_1)-p_\lambda(t_2)|(\Om) \leq \int_{t_1}^{t_2}|\dot p_\lambda(s)|(\Om)ds\leq C(t_2-t_1).$$ 
Applying Ascoli-Arzela Theorem, we extract a further subsequence (independent of time) and find $p \in C^{0,1}([0,T];\mathcal M(\Om;\mathbb M^n_{\rm sym}))$ such that for all $t \in [0,T]$,
$$p_\lambda(t) \rightharpoonup p(t) \quad \text{ weakly* in }\mathcal M(\Om;\mathbb M^n_{\rm sym}).$$
		
Using the additive decomposition $Eu_\lambda=e_\lambda+p_\lambda$ in $\Om$, the previously established weak convergences show that $u \in C^{0,1}([0,T];BD(\Om))$ and, for all $t \in [0,T]$,
$$u_\lambda(t) \rightharpoonup u(t) \quad \text{ weakly* in }BD(\Om).$$		
		
It is now possible to pass to the limit in the initial condition  
$$ u(0) = u_0, \quad  \dot{u}(0) = v_0,\quad  e(0) = e_0, \quad p(0) = p_0,$$
in the additive decomposition:  for all $t \in [0, T]$,
$$E u(t) = e(t) + p(t) \quad \textrm{in $\mathcal M(\Om;\mathbb M^n_{\rm sym})$},$$
and in the equation of motion
$$\ddot{u} - {\rm{div}}\sigma =f \quad \text{ in } L^2 (0,T;L^2(\Om; \R^n)).$$
The stress constraint being convex, hence closed under weak $L^2(\Om;\mathbb M^n_{\rm sym})$ convergence, we further obtain that  for every $t \in [0, T ]$, $\sigma(t) \in \mathbf{K}$ a.e in $\Om$.
	
It remains to show the Neumann boundary condition $\sigma \nu = 0$ on $\Gamma_N$.  Since $\sigma_\lambda \rightharpoonup \sigma$ weakly in $L^2(0,T;H(\Div,\Om))$, we deduce that $\sigma_\lambda\nu \rightharpoonup \sigma\nu$ weakly in $L^2(0,T;H^{-1/2}(\partial\Om;\R^n))$. On the other hand, using estimate \eqref{eq:differenceenergies2}, we have
$$ \|\sigma_\lambda \nu  \|_{L^2(0,T;L^2(\Gamma_N;\R^n))} \leq  \frac {C}{\sqrt\lambda} \to 0,$$
as $\lambda\to \infty$, hence  $\sigma \nu=0$ in $L^2(0,T;L^2(\Gamma_N;\R^n))$.
\end{proof}

\subsection{Flow rule}
It remains to prove the flow rule, which will be performed by passing to the limit in the energy balance obtained in the Theorem \ref{theorem:wellposednesslevelS}, namely, for all $t\in [0,T]$, 
\begin{multline}
\frac{1}{2} \int_{\Om}{ \abs{\dot{u}_\lbd (t)}^2\, dx} + \into  {Q}(e_\lbd(t))\, dx  + \int_0^t  H (\dot p_\lbd(s))(\Om)\,ds +\int_0^t \int_{\partial \Om}  \psi_\lbd(x, \dot{u}_\lbd)\, d\Hs^{n-1}\, ds \\
\leq \frac{1}{2} \int_{\Om}{\abs{v_0}^2 \, dx} + \into  {Q}(e_0) \, dx + \int_0^t \into{f\cdot \dot{u}_\lbd \, dx}\, ds .\label{eq:inqdifferenceenergies}
\end{multline}
The first two terms will easily pass to the lower limit by lower semicontinuity of the norm with respect to weak $L^2$-convergence. The main issue is to pass to the (lower) limit in both last terms in the left-hand side of the previous inequality. The following  result will enable one to obtain a lower bound.
\begin{lemma}
\label{prop:relaxationdirichletpart}
Let $\lk (\hat u_\lbd, \hat e_\lbd, \hat p_\lbd) \rk_{\lambda>0} \subset [BD(\Omega) \cap L^2(\Omega;\R^n)] \times L^2(\Omega;\mathbb M^n_{\rm sym}) \times \mathcal M(\Omega;\mathbb M^n_{\rm sym})$ be such that $E\hat u_\lambda=\hat e_\lambda+\hat p_\lambda$ in $\Om$, 
and
$$
\begin{cases}
\hat u_\lambda \rightharpoonup \hat u & \text{ weakly in } L^2(\Om;\R^n),\\
\hat u_\lambda \rightharpoonup \hat u & \text{ weakly* in } BD(\Om),\\
\hat e_\lambda \rightharpoonup \hat e & \text{ weakly in } L^2(\Om;\mathbb M^n_{\rm sym}),\\
\hat p_\lambda \rightharpoonup \hat p & \text{ weakly* in }\mathcal M( \Om ;\mathbb M^n_{\rm sym}),
\end{cases}
$$
as $\lambda \to \infty$, for some  $(\hat u, \hat e, \hat p) \in [BD(\Omega) \cap L^2(\Omega;\R^n)] \times L^2(\Omega;\mathbb M^n_{\rm sym}) \times \mathcal M(\Om;\mathbb M^n_{\rm sym})$. Then, 
\begin{equation}
\label{eq:rd1}
  H(\hat p)(\Omega)+\int_{{\Gamma_D}}  H(-\hat u \odot \nu)\, d\mathcal H^{n-1}\le \liminf_{\lbd \rightarrow \infty}{\left(  H(\hat p_\lambda)(\Omega) + \int_{\partial \Om}{ \psi_\lbd(x,\hat u_\lbd)\, d \Hs^{n-1}} \right)}. 
\end{equation} 
\end{lemma}
	
\begin{proof}
Without loss of generality, we assume that the right hand side of (\ref{eq:rd1}) is finite. Let $(\lambda_k)_{k\in \N}$ be such that $\lbd_ k \nearrow \infty$ and
$$\liminf_{\lbd \rightarrow  \infty}{\left(  H(\hat p_\lambda)(\Omega) + \int_{\partial \Om}{ \psi_\lbd(x,\hat u_\lbd)\, d \Hs^{n-1}} \right)}= \lim_{k \rightarrow  \infty}{\left( H(\hat p_{\lbd_k})(\Omega)+ \int_{\partial \Om}{ \psi_{\lbd_k}(x,\hat u_{\lbd_k})\, d \Hs^{n-1}} \right)}.$$
As a consequence, there exists a constant $c>0$ (independent of $k$) such that 
$$\int_{\partial \Om}{ \psi_{\lbd_k}(x,\hat u_{\lbd_k})\, d \Hs^{n-1}}\leq c$$ 
for all $k\in \mathbb N$. By definition \eqref{eq:psi} of $\psi_\lambda$ (see also \cite[Lemma 4.9]{BC}), there exists a function $v_{k} \in L^2 (\partial \Om; \R^n)$ such that
\begin{eqnarray*}
\int_{\partial \Om}{\psi_{\lbd_k}(x,\hat u_{\lbd_k})\, d \Hs^{n-1}} & =& \frac{1}{2} \int_{\partial \Om} S_{\lambda_k}(\hat u_{\lbd_k} - v_k)\cdot (\hat u_{\lbd_k} - v_k)\,  d \Hs^{n-1} +  \int_{\partial \Om}{H(-v_{k} \odot \nu ) \, d \Hs^{n-1}}\\
& \geq & \frac{\lambda_k}{2} \int_{\Gamma_D}|\hat u_{\lbd_k} - v_k|^2 \,  d \Hs^{n-1} +  \int_{\partial \Om}{ H(-v_{k} \odot \nu ) \, d \Hs^{n-1}}.
\end{eqnarray*}
By nonnegativity of $H$, we infer that $\hat u_{\lbd_k} - v_{k}\to 0$ in $L^2(\Gamma_D;\R^n)$ as $k \to \infty$. Moreover  
\begin{align}
&H(\hat p_{\lbd_k})(\Omega) + \int_{\partial \Om}{ \psi_{\lbd_k}(x,\hat u_{\lbd_k})\, d \Hs^{n-1}} \nonumber  \\
& \qquad \ge  H(\hat p_{\lbd_k})(\Omega)  +  \int_{\Gamma_D}{ H(-v_{k} \odot \nu)\, d \Hs^{n-1}} \nonumber \\ 
& \qquad \ge  H_\mu(\hat p_{\lbd_k})(\Omega) +  \int_{{ \Gamma_D}}{ H_\mu(-v_{k} \odot \nu)\, d \Hs^{n-1}},\label{eq:fre1}
\end{align}
where $H_\mu:\mathbb M^n_{\rm sym} \to \R^+$ is the Moreau--Yosida transform of $H$ (see \cite[Lemma 1.61]{AFP} or \cite[Lemma 5.30]{FL}), defined by
$$H_\mu(p) := \inf_{q \in \mathbb M^n_{\rm sym} }\{H(q) + \mu \abs{p-q}\} \quad \text{ for all }p \in \mathbb M^n_{\rm sym} .$$
We recall that $H_\mu$ of $H$ enjoys the following properties:
\begin{enumerate}
\item For all $\mu > 0$ we have that $H_\mu \le H$;
\item The function $H_\mu$ is $\mu$-Lipschitz;
\item The function $H_\mu$ is convex as the inf-convolution between the proper convex functions $H$ and $\mu| \cdot |$ (see e.g. \cite[Theorem 5.4]{R});
\item For all $p \in \mathbb M^n_{\rm sym}$, $H_\mu (p) \to H(p)$  as $\mu \rightarrow \infty$.
\end{enumerate}
		
By the $\mu$-Lipschitz continuity of $H_\mu$, adding and subtracting the term $ \int_{\Gamma_D}{H_\mu (-\hat u_{\lbd_k} \odot \nu)\, d \Hs^{n-1}}$ in (\ref{eq:fre1}) yields
\begin{align}
& H(\hat p_{\lbd_k})(\Omega)+ \int_{\partial \Om}{\psi_{\lbd_k}(x,\hat u_{\lbd_k})\, d \Hs^{n-1}} \nonumber \\ 
& \ge  H_\mu(\hat p_{\lbd_k})(\Omega)+ \int_{{\Gamma_D}}{H_\mu(-\hat u_{\lbd_k} \odot \nu)\, d \Hs^{n-1}} -\mu \int_{ \Gamma_D}{\abs{\hat u_{\lbd_k} - v_{k}}  \, d\Hs^{n-1}} .\label{eq:fr3}
\end{align}
Passing to the limit as $k \to \infty$ in \eqref{eq:fr3}, we obtain
\begin{align}
&\lim_{k \rightarrow \infty}{\left( H(\hat p_{\lbd_k})(\Omega) + \int_{\partial \Om}{\psi_{\lbd_k}(x,\hat u_{\lbd_k})\, d \Hs^{n-1}} \right)} \nonumber \\
& \quad \ge \liminf_{k \rightarrow \infty}{\left(  H_\mu(\hat p_{\lbd_k})(\Omega)+ \int_{\Gamma_D}{ H_\mu(-\hat u_{\lbd_k} \odot \nu)\, d \Hs^{n-1}} \right)}.		\end{align}
Let $U \subset \R^N$ be an open set such that $\Gamma_D=U \cap \partial\Om$, and let $\tilde \Om:=\Om \cup U$. We extend $(\hat u_\lambda,\hat e_\lambda,\hat p_\lambda)$ to $\tilde \Om$ as
$$\tilde u_\lambda:=
\begin{cases}
\hat u_\lambda & \text{ in }\Om,\\
0 & \text{ in }\tilde \Om \setminus \Om,
\end{cases}
\qquad
\tilde e_\lambda:=
\begin{cases}
\hat e_\lambda & \text{ in }\Om,\\
0 & \text{ in }\tilde \Om \setminus \Om,
\end{cases}
$$
and
$$\tilde p_\lambda:=E\tilde u_\lambda-\tilde e_\lambda=\hat p_\lambda \res \Om - \hat u_\lambda \odot \nu \mathcal H^{n-1} \res {\Gamma_D}.$$
Similarly, we set 
$$\tilde u:=
\begin{cases}
\hat u & \text{ in }\Om,\\
0 & \text{ in }\tilde \Om \setminus \Om,
\end{cases}
\qquad
\tilde e:=
\begin{cases}
\hat e & \text{ in }\Om,\\
0 & \text{ in }\tilde \Om \setminus \Om.
\end{cases}
$$
	
Note that $\tilde p_\lambda \rightharpoonup \tilde p$ weakly* in $\mathcal M(\tilde\Om;\mathbb M^n_{\rm sym})$ with $\tilde p=E\tilde u-\tilde e=\hat p \res \Om - \hat u \odot \nu \mathcal H^{n-1} \res { \Gamma_D}$. Using that $H_\mu$ is a continuous, convex and positively one homogeneous function with $H_\mu(0)=0$, we can apply Reshetnyak's lower semicontinuity Theorem (see \cite[Theorem 2.38]{AFP}) to get that
\begin{multline*}
\liminf_{k \rightarrow \infty}{\left(  H_\mu(\hat p_{\lbd_k})(\Omega)+ \int_{{ \Gamma_D}}{H_\mu(-\hat u_{\lbd_k} \odot \nu)\, d \Hs^{n-1}} \right)}\\
=\liminf_{k \rightarrow \infty}  H_\mu (\tilde p_{\lbd_k})(\tilde\Omega) \geq  H_\mu (\tilde p)(\tilde \Omega)\\
=  H_\mu (\hat p)(\Omega)+ \int_{{ \Gamma_D}}{ H_\mu(-\hat u \odot \nu)\, d \Hs^{n-1}}.
\end{multline*}
We have thus estabished that for all $\mu>0$,
$$\liminf_{\lbd \rightarrow  \infty}{\left(  H(\hat p_\lbd)(\Omega) + \int_{\partial \Om}{\psi_\lbd(x,\hat u_\lbd)\, d \Hs^{n-1}} \right)}
\geq   H_\mu (\hat p)(\Omega) + \int_{{ \Gamma_D}}{H_\mu(-\hat u \odot \nu)\, d \Hs^{n-1}}.$$
We can now pass to the limit as $\mu \to \infty$ owing to the Monotone Convergence theorem to get that
$$\liminf_{\lbd \rightarrow  \infty}{\left(  H(\hat p_\lbd)(\Omega) + \int_{\partial \Om}{\psi_\lbd(x,\hat u_\lbd)\, d \Hs^{n-1}} \right)} 
\geq  H (\hat p)(\Omega) + \int_{{ \Gamma_D}}{H(-\hat u \odot \nu)\, d \Hs^{n-1}},$$
which leads to the desired lower bound.
\end{proof}
	
We are now in position to prove a lower bound energy inequality. Since for all $t \in [0,T]$, we have $\dot u_\lambda(t) \rightharpoonup \dot u(t)$ weakly in $L^2(\Om;\R^n)$ and $e_\lambda(t) \rightharpoonup e(t)$ weakly in $L^2(\Om;\mathbb M^n_{\rm sym})$, we get by weak lower semicontinuity of the norm that
$$\frac{1}{2} \int_{\Om}{|\dot{u}(t)|^2 \, dx} + \mathcal{Q}(e(t))\leq \liminf_{\lambda \to \infty} \left\{\frac{1}{2} \int_{\Om}{\abs{\dot{u}_\lbd (t)}^2\, dx} + \mathcal{Q}(e_\lbd(t))\right\}.$$
To pass to the lower limit in the remaining terms in the left-hand side of the energy inequality \eqref{eq:inqdifferenceenergies}, we consider a partition $0=t_0 \le t_1 \le \ldots \le t_N = t$ of the time interval $[0, t]$. By convexity of $H$ and $\psi_\lambda(x,\cdot)$, we infer from Jensen's inequality that
\begin{multline*}
\int_0^t    H(\dot{p}_\lbd(s))(\Omega) \, ds +   \int_0^t \int_{\partial \Om}    \psi_\lbd(x, \dot{u}_\lbd (s))\, d\Hs^{n-1}\, ds \\
\ge \sum_{i =1}^{N}  \left\{  H(p_\lbd(t_i) - p_\lbd(t_{i-1}) )(\Omega) + \int_{\partial \Om} \psi_\lbd(x, u_\lbd (t_{i})-u_\lbd (t_{i-1}))\, d \Hs^{n-1}\right\}.
\end{multline*}
Since, for all $0 \le i \le N$ we have that 
$$
\begin{cases}
u_\lambda(t_i) \rightharpoonup  u(t_i) & \text{ weakly in } L^2(\Om;\R^n),\\
u_\lambda(t_i) \rightharpoonup  u(t_i) & \text{ weakly* in } BD(\Om),\\
e_\lambda(t_i) \rightharpoonup e(t_i) & \text{ weakly in } L^2(\Om;\mathbb M^n_{\rm sym}),\\
p_\lambda(t_i) \rightharpoonup p(t_i) & \text{ weakly* in }\mathcal M(\Om;\mathbb M^n_{\rm sym}),
\end{cases}
$$
we can apply Proposition \ref{prop:relaxationdirichletpart} to get that
$$ \liminf_{\lbd \rightarrow \infty}\left(\int_0^t   H(\dot{p}_\lbd(s))(\Omega)\, ds +   \int_0^t \int_{\partial \Om}   \psi_\lbd(x, \dot{u}_\lbd)\, d\Hs^{n-1}\, ds\right) \ge \sum_{i =1}^{N}  H(p(t_i) - p(t_{i-1}) )(\Om \cup \Gamma_D),$$
where the measure $p(t)$ is extended to $\Gamma_D$ by setting 
$$p(t)\res\Gamma_D := -u(t)\odot\nu \mathcal H^{n-1}\res\Gamma_D.$$

Passing to the supremum with respect to all partitions, we deduce that
$$\mathcal V_{\mathcal H}(p;0,t):=\sup\left\{\sum_{i=1}^N H(p(t_i) - p(t_{i-1}) )(\Om \cup \Gamma_D) :\; 0=t_0 \leq t_1 \leq \cdot \leq t_N=t, \, N \in \mathbb N\right\}<\infty.$$
Using \cite[Theorem 7.1]{DMDSM}\footnote{Note that \cite[Theorem 7.1]{DMDSM} is stated for functions $H$ which are bounded from above, which is not our case here because $H$ is allowed to take the value $+\infty$.  However, a carefull inspection of the proof of \cite[Theorem 7.1]{DMDSM} shows the validy of this result in our case thanks to the additional property $\mathcal V_{\mathcal H}(p;0,t)<\infty$.}, we get that 
$$ \liminf_{\lbd \rightarrow \infty}\left(\int_0^t  H(\dot{p}_\lbd(s))(\Omega)\, ds +   \int_0^t \int_{\partial \Om}  \psi_\lbd(x, \dot{u}_\lbd)\, d\Hs^{n-1}\, ds\right) \ge \int_0^t   H(\dot p(s) )(\Om \cup \Gamma_D)\, ds.$$
Passing to the lower limit in \eqref{eq:inqdifferenceenergies} as $\lambda \to \infty$ yields
\begin{multline}\label{eq:ineq1}
\frac{1}{2} \int_{\Om}{|\dot{u} (t)|^2 \, dx} + \mathcal{Q}(e(t))+\int_0^t  H(\dot p(s) )(\Om \cup \Gamma_D) \, ds \\
\leq \frac{1}{2} \int_{\Om}{ \abs{v_0}^2 \, dx} + \mathcal Q(e_0) + \int_0^t \into{ f\cdot \dot{u}\, dx}\, ds.
\end{multline}
	
The proof of the other energy inequality relies on the convexity inequality proved in Section \ref{sec:duality}. Indeed, assuming one of the following assumptions:
\begin{itemize}
\item $\partial\Om=\Gamma_D$;
\item $\partial\Om=\Gamma_N$;
\item $n=2$ and $\Sigma$ is a finite set;
\item $n=3$ and $\mathbf K=K_D \oplus (\R\, {\rm Id})$, for some compact and convex set $K_D \subset \mathbb M^3_D$ containing $0$ in its interior;
\end{itemize}
we can appeal Proposition \ref{prop:an1}, Proposition \ref{prop:n=2} or Proposition \ref{prop:n=3}. Indeed, for a.e. $t \in [0,T]$, we have $(\dot u(t),\dot e(t),\dot p(t)) \in \mathcal A_0$,  $\sigma(t) \in \mathcal K \cap \mathcal S_0$ and $H(\dot p(t))$ is a finite measure (by \eqref{eq:ineq1}). As a consequence, for a.e. $t \in [0,T]$, the duality pairing $[\sigma(t) \colon \dot p(t)] \in \mathcal D'(\R^n)$ is well defined and it extends to a bounded Radon measure supported in $\overline \Omega$ with
\begin{equation}
H(\dot p(t)) \geq [\sigma(t) \colon \dot p(t)]\quad\text{in }\mathcal M(\R^n)\,. \label{ea:reverseinequality}
\end{equation}
	
Since the nonnegative measure $H(\dot p(t)) - [\sigma(t) \colon \dot p(t)]$ is compactly supported in $\overline\Om$, we can evaluate its mass by taking the test function $\varphi \equiv 1$ in Definition \ref{definition:dualityMixed}. We then obtain that for a.e. $t \in [0,T]$,
$$0 \leq H(\dot p(t))(\Om \cup \Gamma_D) +\int_\Omega  \sigma(t) : \dot e(t)\,  dx +  \int_\Omega \dot u(t) \cdot {\rm div} \sigma(t) \, dx.$$
Using the equation of motion and the regularity properties of $\dot u$ and $e$, we can integrate by parts respect to time and get that 
\begin{multline*}
0 \leq \int_0^t H(\dot p(s))(\Om \cup \Gamma_D) )\, ds+ \mathcal Q(e(t)) -\mathcal Q(e_0)\\
+\frac12 \int_\Om |\dot u(t)|^2 \, dx - \frac12 \int_\Om |v_0|^2 \, dx -\int_0^t \int_\Om f\cdot u\, dx\, ds.
\end{multline*}
Owing to the first energy inequality \eqref{eq:ineq1}, we deduce that the last expression is zero, which implies that the nonnegative measure $H(\dot p(t)) - [\sigma(t) \colon \dot p(t)]$ has zero mass in $\overline\Om$. This leads in turn that this measure vanishes in $\overline \Om$, in other words the flow rule $H(\dot p(t)) = [\sigma(t) \colon \dot p(t)]$ in $\mathcal M(\overline \Om)$ is satisfied. Finally, since $H(\dot p(t))$ is concentrated on $\Omega \cup \Gamma_D$, it follows that $[\sigma(t):\dot p(t)]$ vanishes on $\partial\Omega \setminus \Gamma_D$ and that the flow rule $H(\dot p(t))=[\sigma(t):\dot p(t)]$ holds in $\mathcal M(\Omega \cup \Gamma_D)$.

\subsection{Uniqueness}
	
Let $(u_1, e_1,p_1)$ and $(u_2,e_2,p_2)$ be two solutions given by Theorem \ref{thm:compactness}. Subtracting the equations of motion of each solution, we have
$$\ddot{u}_1 - \ddot{u}_2 - \Div(\sigma_1 - \sigma _2) = 0 \quad \textrm{in } L^2(0,T; L^2 (\Om; \R^n)).$$
Let us consider the test function $\varphi := {\bf 1}_{[0,t]} (\dot{u}_1 - \dot{u}_2) \in L^2(0,T; L^2(\Om; \R^n))$, we deduce
\begin{equation}
\label{eq:uniqueness1}
\intzt{\into{(\ddot{u}_1 - \ddot{u}_2):(\dot{u}_1 - \dot{u}_2)\,dx}\,ds} - \intzt{\into{ (\Div(\sigma_1 - \sigma _2)) \cdot (\dot{u}_1 - \dot{u}_2)\, dx }\, ds} = 0.
\end{equation}
Since $\ddot{u}_1 - \ddot{u}_2 \in L^2(0,T; L^2 (\Om; \R^n))$ and $\dot{u} _1 (0) = \dot{u} _2 (0) = v_0$, we infer that
\begin{equation}\label{eq:1406}
\intzt{\into{(\ddot{u}_1 (s) - \ddot{u}_2(s)):(\dot{u}_1 (s)- \dot{u}_2 (s))\,dx} \,ds} = \frac{\norm{\dot{u}_1 (t) - \dot{u}_2 (t) }^2_{L^2(\Om; \R^n)}  }{2}.
\end{equation}
We already know that, for a.e. $s \in [0,T]$, the distributions $[\sigma_1(s):\dot p_1(s)]$ and 	$[\sigma_2(s):\dot p_2(s)]$ belong to $\mathcal M(\Omega \cup \Gamma_D)$. Moreover, since $(\dot u_1(s),\dot e_1(s),\dot p_1(s)),\, (\dot u_2(s),\dot e_2(s),\dot p_2(s)) \in \mathcal A_0$, $\sigma_1(s), \,\sigma_2(s) \in \mathcal S_0 \cap \mathcal K$ and $H(\dot p_1(s)), \, H(\dot p_2(s))$ are finite measures we can appeal Propositions \ref{prop:an1}, \ref{prop:n=2} and \ref{prop:n=3} which state that  $[\sigma_2(t) \colon \dot p_1(s)]$ and $ [\sigma_1(s) \colon \dot p_2(s)]$ extend to bounded Radon measures supported in $\overline \Omega$ with
$$[\sigma_1 (s): \dot{p}_1 (s)] = {H} (\dot{p}_1 (s)) \ge [\sigma_2 (s): \dot{p}_1 (s)] \quad \text{ in }\mathcal M(\R^n),$$
and
$$ [\sigma_2 (s): \dot{p}_2 (s)]  = {H} (\dot{p}_2 (s)) \ge  [\sigma_1 (s): \dot{p}_2 (s)] \quad \text{ in }\mathcal M(\R^n).$$
As a consequence, the measure $ [(\sigma_1(s) - \sigma_2(s)) : (\dot{p}_1(s) - \dot{p}_2(s))] $ is nonnegative. Furthermore, by the definition of stress duality (see Definition \ref{definition:dualityMixed}  with the test function $\varphi\equiv 1$ and $g=0$), we infer that
\begin{eqnarray}
0 & \leq & \int_0^t  [(\sigma_1(s) - \sigma_2(s)) : (\dot{p}_1(s) - \dot{p}_2(s))](\overline\Omega)\nonumber\\
& = & -\intzt{\into{ (\sigma_1(s) - \sigma_2(s)) : (\dot{e}_1(s) - \dot{e}_2(s)) \, dx }\, ds}\nonumber\\
&&\qquad\qquad- \intzt{\into{ (\Div(\sigma_1 (s) - \sigma _2(s))) \cdot (\dot{u}_1(s) - \dot{u}_2(s)) \, dx }\, ds}\nonumber\\
& = & - \mathcal{Q}(e_1 (t) - e_2 (t) )- \intzt{\into{ (\Div(\sigma_1 (s) - \sigma _2(s))) \cdot (\dot{u}_1(s) - \dot{u}_2(s)) \, dx }\, ds} , \label{eq:uniqueness2}
\end{eqnarray}
where we have used the fact that $e_1 (0) = e_2(0) = e_0$. By \eqref{eq:uniqueness1}, \eqref{eq:1406} and \eqref{eq:uniqueness2}, we infer that
$$ \frac{\norm{\dot{u}_1 (t) - \dot{u}_2 (t) }^2_{L^2(\Om; \R^n)}  }{2} +  \mathcal{Q}(e_1 (t) - e_2 (t) ) \le 0.  $$
From the expression above, we infer that $e_1 = e_2$ and $\dot u_1=\dot u_2$.  Since, $u_1 (0) = u_2 (0) = u_0$, we conclude that $u_1 = u_2$, and by the kinematic compatibility $p_1 = p_2$. This concludes the proof of the uniqueness. In particular, by uniqueness of the limit, there is no need of extracting subsequences when passing to the limit as $\lambda \to \infty$. The proof of Theorem \ref{thm:compactness} is now complete.

\subsection*{Acknowledges} R. Llerena acknowledges support from the Austrian Science Fund (FWF) through projects P 29681 and TAI 293-N, and from BMBWF through the OeAD-WTZ project HR 08/2020. This work was supported by a public grant as part of the Investissement d'avenir project, reference ANR-11-LABX-0056-LMH, LabEx LMH. 
	
\nocite{*}
\bibliographystyle{amsplain}
\bibliography{reference1.bib}
	
\end{document}